\documentclass[12pt]{amsart}

\usepackage{amssymb}

\usepackage{cleveref}

\newtheorem{thm}{Theorem}

\newtheorem{lem}[thm]{Lemma}
\newtheorem{cor}[thm]{Corollary}
\newtheorem{prop}[thm]{Proposition}

\theoremstyle{definition}
\newtheorem{defn}[thm]{Definition}
\newtheorem{notn}[thm]{Notation}

 \newcounter{case}

 \renewcommand{\thecase}{\arabic{case}}

\newcounter{subcase}

 \renewcommand{\thesubcase}{\alph{subcase}}

\numberwithin{thm}{section}

\usepackage{fullpage}

\def\Aut{{\rm Aut}}
\def\Z{{\mathbb Z}}
\def\Cay{{\rm Cay}}
\def\Sym{{\rm Sym}}
\def\lcm{{\rm lcm}}

\title{Dihedral groups of order $2pq$ or $2pqr$ are DCI}

\author{Joy Morris}

\thanks{Supported by the Natural Science and Engineering Research Council of Canada (grant RGPIN-2017-04905).}

\address{Department of Mathematics and Computer Science\\
	University of Lethbridge\\
	Lethbridge, AB T1K 3M4\\
	Canada}

\email{joy.morris@uleth.ca}



\begin{document}



\maketitle

\begin{abstract}
A group has the (D)CI ((Directed) Cayley Isomorphism) property, or more commonly is a (D)CI group, if any two Cayley (di)graphs on the group are isomorphic via a group automorphism. That is, $G$ is a (D)CI group if whenever $\Cay(G,S)\cong \Cay(G,T)$, there is some $\delta \in \Aut(G)$ such that $S^\delta=T$. (For the CI property, we only require this to be true if $S$ and $T$ are closed under inversion.)

Suppose $p,q,r$ are distinct odd primes. We show that $D_{2pqr}$ is a DCI group. 
We present this result in the more general context of dihedral groups of squarefree order; some of our results apply to any such group, and may be useful in future toward showing that all dihedral groups of squarefree order are DCI groups.
\end{abstract}

\section{Introduction}

The Cayley Isomorphism (CI) and Directed Cayley Isomorphism problems for groups and graphs are long-standing problems of interest to algebraic graph theorists. The standard formulation for these problems and the basic tools used in proving them date back to~\cite{Babai}. Special cases of the problem (particularly for cyclic groups) had been studied prior to Babai's paper, but he presented them in a uniform context with helpful terminology and provided tools that have been essential to much of the work that has followed.

Let $G$ be a group and $S \subseteq G$. We define the Cayley (colour) (di)graph $\Cay(G,S)$ to be the (colour) (di)graph whose vertices are the elements of $G$, with an arc from the vertex $g$ to the vertex $sg$ if and only if $s \in S$. For colour (di)graphs, each element of $S$ has an associated colour, and the arcs that arise using that element of $s$ are given that colour. Note that graph automorphisms coming from elements of $G$ will be acting by multiplication on the right. We will use exponents to denote the actions of group automorphisms and action by conjugation, and write other permutation group actions on sets on the right but without an exponent, as the details of our proofs would get very difficult to read in the exponents.

The Cayley (di)graph $\Cay(G,S)$ has the (D)CI ((Directed) Cayley Isomorphism) property, or more commonly is a (D)CI graph, if whenever $\Cay(G,S)\cong \Cay(G,T)$, there is some $\delta \in \Aut(G)$ such that $S^\delta=T$. For a Cayley colour (di)graph, both the isomorphism and the group automorphism must preserve the colours that have been assigned to the elements of $S$ and $T$. A group has the CI property if every Cayley graph on the group has the CI property. It has the DCI property if every Cayley digraph on the group has the DCI property. It has the CI$^{(2)}$ property if every Cayley colour digraph on the group has the DCI property (this notation comes from the 2-closure of a group, which will arise later in this paper). If a group has the DCI property, then since every Cayley graph is also a digraph (each edge is equivalent to a digon of arcs), it also has the CI property. Likewise, if it has the CI$^{(2)}$ property then it has the DCI property. Although our results in this paper and many of the results we discuss in fact prove that groups are CI$^{(2)}$ groups, Cayley colour digraphs are not much studied and this terminology is less common, so we will refer to the DCI property and DCI groups throughout the remainder of this paper, except in stating Babai's criterion. Since we prove that Babai's criterion holds, our result does in fact show that these dihedral groups are CI$^{(2)}$ groups.

Much work by many authors has gone into the study of the (D)CI properties, and the groups that can be CI groups are quite limited. In particular, if a group is (D)CI then so is every subgroup (and every quotient). Given that whenever $p$ is an odd prime, the cyclic group $\mathbb Z_{p^2}$ is not DCI, and elementary abelian $p$-groups of rank at least $2p+3$ are not DCI groups, groups of squarefree order are a significant aspect of this problem. For cyclic groups, the DCI problem was completely solved by Muzychuk~\cite{Ma,Muz2}. 

\begin{thm}[Muzychuk~\cite{Ma,Muz2}]\label{Muz-cyclic}
A cyclic group is a DCI group if and only if its order is either squarefree, or twice a squarefree number.
\end{thm}

Our main result is the following.

\begin{thm}\label{main}
Suppose $p,q,r$ are distinct odd primes. Then $D_{2pq}$ and $D_{2pqr}$ are DCI groups.
\end{thm}

Although we are only able to complete the proof for $3$ odd primes, we will set up our notation and prove some of our results in the more general context in which the dihedral group has order divisible by an arbitrary number of odd primes, in hopes that these results may be useful in future to prove that dihedral groups with more prime factors also have the DCI property.

It was shown in \cite{DMS} that $D_{6p}$ is a DCI group. The $D_{2pq}$ part of our theorem is a generalisation of that work. 

In 2002, Dobson~\cite{Dobson2002} worked on the CI problem for dihedral groups, and was able to show that $D_{2n}$ is a DCI group under some fairly strong conditions (this result is somewhat obscured by technical definitions, but is Theorem 22). His result required that $n$ be odd and squarefree, and that $\gcd(n,\varphi(n))=1$. He also assumed that if $n=p_1\cdots p_s$ where $p_1<\ldots <p_s$ are distinct odd primes, then for each $2 \le i \le s$, $p_i > 2p_1\cdots p_{i-1}$. However, he used this final hypothesis only to ensure the existence of many $G$-invariant partitions (this will be discussed further a bit later). With the new result~\Cref{blocks} found in \cite{DMS-CI3} to provide such $G$-invariant partitions, this hypothesis can be dispensed with. In addition to explicitly dispensing with the hypothesis that~\cite{DMS-CI3} shows to be unnecessary, our result dispenses with Dobson's hypothesis that $\gcd(n,\varphi(n))=1$.

The main tool Babai provided in \cite{Babai} is based on the automorphism group, and can be used to determine whether or not a graph is a (D)CI graph. In fact, it can be used to understand whether or not  every Cayley colour (di)graph has the (D)CI property.

\begin{lem}[Babai, \cite{Babai}]
Let $R$ be a finite group and let $S \subseteq R$. Then $\Cay(R,S)$ is a DCI graph if and only if for any $R' \le \Aut(\Cay(R,S))$ with $R' \cong R$, there is some $\delta \in \Aut(\Cay(R,S))$ such that $(R')^\delta=R$.
\end{lem}

Note that $(R')^\delta=\delta^{-1}R'\delta$.

In order to use this concept most effectively to determine that a group has the CI property, we require the concept of the 2-closure of a permutation group. This concept was studied in some detail in the works of Wielandt~\cite{Wielandt}.

\begin{defn}
Let $G$ be a permutation group acting on a finite set $\Omega$. The \emph{2-closure} of $G$, denoted $G^{(2)}$, is the smallest permutation group containing $G$ that can be the automorphism group of a digraph. More precisely, 
$$G^{(2)}=\{\beta \in \Sym(\Omega) : \forall (x,y)\in \Omega^2, \exists g_{x,y} \in G \text{ with }(x,y)\beta=(x,y){g_{x,y}}\}.$$
\end{defn}

This leads us to the following standard characterisation of CI$^{(2)}$ groups based on Babai's result.

\begin{lem}[Standard, based on Babai]\label{Babai-conj}
Let $R$ be a finite group and let $R_r$ be the right-regular representation of $R$ in $\Sym(R)$. The groups $R_r$ and $R_r^\pi$ are conjugate in $\langle R_r,R_r^\pi\rangle^{(2)}$ for every $\pi \in \Sym(R)$ if and only if $R$ is a CI\/$^{(2)}$ group.
\end{lem}

\section{Preliminaries: Notation and $G$-invariant partitions}

For the purposes of this paper, $R$ will be dihedral of squarefree order, say $2k$, where $k$ is odd and squarefree. Any dihedral groups that could potentially have the DCI property have this structure. Although the most natural generating set for $R$ has two elements (one of order $k$ and the other of order $2$), it will prove much easier to work with if we use one generator for each prime divisor. 

\begin{notn}\label{notn-1}
Henceforth in this paper, we use the following notation:
\begin{itemize}
\item $p_1, p_2, \ldots, p_s$ are distinct primes;
\item $R_r=\langle \rho_1,\ldots, \rho_s,\tau_1\rangle$, where $|\rho_i|=p_i$ for each $1 \le i \le s$, and $|\tau_1|=2$;
\item $R_r^\pi=\langle \sigma_1,\ldots,\sigma_s,\tau_2\rangle$ with $|\sigma_i|=p_i$ for each $1 \le i \le s$, and $|\tau_2|=2$; 
\item  both $R_r$ and $R_r^\pi$ are permutation groups acting regularly on the set $\Omega$ of cardinality $2p_1\cdots p_s$;
\item $G=\langle R_r,R_r^\pi \rangle$.
\end{itemize}
\end{notn}

Our goal will be to show that there is some $\beta \in G^{(2)}$ such that $R_r^{\pi\beta}=R_r$.

In this paper, we will sometimes simplify our notation with an abuse: suppose that we can find some $\beta_1 \in \langle R_r,R_r^\pi\rangle^{(2)}$ such that $R_r^{\pi\beta_1}=\beta_1^{-1}R_r^\pi\beta_1$ has some desirable properties and $\langle R_r,R_r^{\pi\beta_1}\rangle^{(2)}\le \langle R_r, R_r^\pi\rangle^{(2)}$. In this event, rather than writing $R_r^{\pi\beta_1}$ thenceforward, we ``replace" $R_r^\pi$ by this new group, and replace each generator in whatever standard generating set we are using for $R_r^\pi$ by the appropriate conjugate under $\beta_1$. In effect, from this point forward we behave as though $R_r^\pi$ had been this new conjugate all along, since we know we can reach this through conjugation in $\langle R_r,R_r^\pi\rangle^{(2)}$. We may do this repeatedly, with a $\beta_2$, etc. We will provide some additional justification that this abuse does not invalidate our proofs, at the end of this section.

For the rest of this section we focus on $G$-invariant partitions, and show that after conjugating $R_r^\pi$ by some element of $G^{(2)}$ if necessary, the resulting $G=\langle R_r, R_r^\pi\rangle$ admits a sequence of nested $G$-invariant partitions: one consisting of $2p_{i+1}\cdots p_s$ blocks of cardinality $p_1\cdots p_i$ for every $1 \le i \le s$. We also show some additional desirable properties that we may assume our partitions have, describe circumstances under which we can reorder our primes while maintaining all of our key hypotheses about partitions, and develop additional notation based on all of this information.

\begin{defn}
Given a transitive group $G$ acting on the set $\Omega$, a partition $\mathcal B$ of $\Omega$ is \emph{$G$-invariant} if for every $B \in \mathcal B$ and every $g \in G$, $Bg \in \mathcal B$. Equivalently, $Bg \cap B \neq \emptyset$ implies that $Bg=B$.

If $|\mathcal B|=a$ and $|B|=b$ for every $B \in \mathcal B$, we say that the partition $\mathcal B$ consists of $a$ blocks of cardinality $b$.

The $G$-invariant partition $\mathcal B$ is \emph{normal} if its blocks are the orbits of a normal subgroup of $G$.
\end{defn}

There are some useful ways to understand $G$-invariant partitions. The next lemma is well-known and easily follows from the definition of $G$-invariant partitions.

\begin{lem}\label{block-form-not-regular}
Suppose that $G$ is a transitive permutation group acting on the set $\Omega$. If $\mathcal B$ is a $G$-invariant partition then given any $y \in \Omega$, the blocks of $\mathcal B$ are the collection $\{y{H\gamma}: \gamma \in G\}$ for some $H \le G$.
\end{lem}

In the situation of regular actions, we get much more information.

\begin{lem}\label{block-form}
Suppose that $G$ is a regular permutation group acting on the set $\Omega$. Then the converse of~\Cref{block-form-not-regular} holds; that is, given any $y \in \Omega$ and any $H \le G$, the collection $\{y{H\gamma}: \gamma \in G\}$ is a $G$-invariant partition.

Accordingly, for each $z \in \Omega$, the block of $\mathcal B$ that contains $z$ is $z{H}$ if and only if $z=y\gamma$ for some $\gamma \in G$ such that $H\gamma=\gamma H$. In particular, the blocks of $\mathcal B$ are the orbits of $H$ if and only if $H \trianglelefteq G$.
\end{lem}

We identify some easy consequences of~\Cref{block-form} that will be useful in the context of this paper.

\begin{lem}\label{orbits-blocks}
Suppose that $H_1 \le R_r$ and $H_2 \le R_r^\pi$, using~\Cref{notn-1}. If for any fixed $x \in \Omega$ and for every $\alpha \in R_r$ there is some $\gamma \in R_r^\pi$ such that $x{H_1\alpha}=x{H_2\beta}$, then the collection $\{x{H_1\alpha}:\alpha \in R_r\}$ is a $G$-invariant partition.

In fact, if $H_1\le R_r$ is a cyclic subgroup of odd order in $R_r$, then $H_1$ has the same orbits on $\Omega$ as some $H_2 \le R_r^\pi$ if and only if the orbits of $H_1$ form a $G$-invariant partition.
\end{lem}

\begin{proof}
By~\Cref{block-form}, since $H_1 \le R_r$, $\{x{H_1\alpha}: \alpha \in R_r\}$ is an $R_r$-invariant partition. Likewise, $\{x{H_2\beta}: \beta \in R_r^\pi\}$ is an $R_r^\pi$-invariant partition. Since these partitions coincide, the partition is
 invariant under both $R_r$ and $R_r^\pi$. As $G=\langle R_r,R_r^\pi\rangle$, it must be invariant under $G$.

Since any cyclic subgroup of odd order in a dihedral group is normal, if $H_1$ is such a subgroup then $H_1 \triangleleft R_r$. If $H_2$ has the same orbits then due to the regular actions of $R_r$ and $R_r^\pi$, we must have $|H_2|=|H_1|$ is odd, so $H_2$ is cyclic and $H_2 \triangleleft R_r^\pi$. Since the orbits of $H_1$ and $H_2$ coincide, by~\Cref{block-form} these orbits form a $G$-invariant partition.

Conversely, if the orbits of $H_1$ form a $G$-invariant partition then they form a $R_r^\pi$-invariant partition. Since these orbits have cardinality $|H_1|$, by~\Cref{block-form-not-regular} they must be the collection $\{y{H_2\beta}: \beta \in R_r^\pi\}$ for some $H_2 \le R_r^\pi$, and furthermore $|H_2|=|H_1|$ is odd. The odd order forces $H_2$ to be a normal cyclic subgroup of $R_r^\pi$, so by~\Cref{block-form} the blocks are the orbits of $H_2$.
\end{proof}

It is always the case (and easy to see) that the intersection of blocks in two $G$-invariant partitions, is a block of a $G$-invariant partition. In our situation with dihedral groups, we can say something similar about combinations of blocks.

\begin{lem}\label{block-union}
Using~\Cref{notn-1}, suppose that $G$ has an invariant partition consisting of $2$ blocks of cardinality $p_1 \cdots p_s$. Suppose also that $\mathcal C$ and $\mathcal D$ are $G$-invariant partitions with blocks of cardinality $a$ and $b$ respectively. Then there is also a $G$-invariant partition with blocks of cardinality $\lcm(a,b)$. A block of this partition can be formed by fixing $C \in \mathcal C$ and taking the union of every $D \in \mathcal D$ such that $D \cap C \neq \emptyset$.
\end{lem}

\begin{proof}
Let $\mathcal F=\{F_1, F_2\}$ be the $G$-invariant partition with $2$ blocks. If the blocks of either $\mathcal C$ or $\mathcal D$ have even cardinality, take their intersections with $F_1$ and $F_2$ to get $G$-invariant partitions $\mathcal C'$ and $\mathcal D'$ the cardinality of whose blocks is the largest odd divisor of the original block cardinality. (Since $|R_r|$ is squarefree, the original cardinality was either odd or twice an odd number, so taking the intersections with $F_1$ and $F_2$ does accomplish this.) By~\Cref{block-form} the blocks of $\mathcal C'$ and $\mathcal D'$ are the orbits of some subgroup of the cyclic subgroup of index $2$ in $R_r$, say $\langle \alpha_1\rangle$ and $\langle \alpha_2\rangle$ where $\alpha_1, \alpha_2 \in R_r$. By~\Cref{orbits-blocks}, we have $\langle \alpha_1 \rangle$ has the same orbits as $\langle \gamma_1\rangle$ and $\langle \alpha_2 \rangle$ has the same orbits as $\langle \gamma_2 \rangle $ for some $\gamma_1, \gamma_2 \in R_r^\pi$.

Now again by~\Cref{block-form}, the orbits of $\langle \alpha_1\alpha_2\rangle$ (a normal subgroup of $R_r$) are invariant under $R_r$, and the orbits of $\langle \gamma_1\gamma_2\rangle$ are invariant under $R_r^\pi$. Since the orbits of $\langle \alpha_1\rangle$ and $\langle \gamma_1 \rangle$ coincide as do the orbits of $\langle \alpha_2 \rangle$ and $\langle \gamma_2 \rangle$, the orbits of $\langle \alpha_1,\alpha_2 \rangle$ and $\langle \gamma_1, \gamma_2 \rangle$ also coincide. So these orbits are invariant under both $R_r$ and $R_r^\pi$ and therefore under $G$. This completes the proof if $a$ and $b$ were odd, since $|\langle \alpha_1,\alpha_2\rangle|=|\alpha_1\alpha_2|=\lcm(a,b)$. If either $a$ or $b$ was even, then the blocks of this partition are half the desired cardinality.

Without loss of generality, suppose $a$ is even. Let $x \in \Omega$. Then there is some $\tau \in R_r$ such that $x\tau$ is in the same block of $\mathcal C$ as $x$. We claim that if $\mathcal E$ is the $G$-invariant partition we just found and $x \in E \in \mathcal E$, then $\{(E \cup E \tau)g: g \in G\}$ is a $G$-invariant partition. Suppose that $g \in G$ and $(E \cup E\tau)\cap (Eg \cup E\tau g) \neq \emptyset$. Without loss of generality, since $\mathcal E$ is $G$-invariant, the only way we can have $E \cup E \tau \neq Eg \cup E\tau g$ is if either $E=Eg$ and $E\tau \cap E\tau g = \emptyset$, or if $E=E\tau g$ and $E \tau \cap Eg = \emptyset$. In the former case, $x \in E=Eg$ so since $x \tau$ is in the same block of $\mathcal C$ as $x$ and this block is fixed by $g$, we have $x \tau \in E\tau \cap E \tau g$, a contradiction. In the latter case, $x \in E=E\tau g$ and $x \tau$ is in the same block of $\mathcal C$ as $x$, and this block is fixed by $\tau g$. Thus $x \tau \tau g=xg$ is in $E\tau$ and $Eg$, again a contradiction. This gives us blocks of twice the previous cardinality, completing the proof.
\end{proof}

Sometimes one partition is a refinement of another; this leads to a partial order on partitions.

\begin{defn}
If $\mathcal B$ and $\mathcal C$ are both partitions of $\Omega$, we say that $\mathcal B \preceq \mathcal C$ if for every $B \in \mathcal B$, there is some $C \in \mathcal C$ such that $B \subseteq C$. In other words, $\mathcal B \preceq \mathcal C$ if each block of $\mathcal C$ is a union of blocks of $\mathcal B$.

If $\mathcal B \preceq \mathcal C$ but $\mathcal B \neq \mathcal C$ then we can write $\mathcal B \prec \mathcal C$.
\end{defn}

We now provide the new result of~\cite{DMS-CI3} that allows us to avoid making assumptions about the relative sizes of the primes dividing the order of our dihedral group. Their result (Corollary 4.6 of their paper) is stated in a broader context. Extracting our statement from their paper requires understanding that by their Definition 1.6, $\mathcal R_n$ includes dihedral groups of squarefree order, and noting that in the situation of dihedral groups of squarefree order, the Sylow $2$-subgroups are isomorphic to $\Z_2$, and therefore have trivial automorphism group, so the set of prime divisors of the order of this automorphism group is empty; this is $\pi(|\Aut(R_2)|)$ in their notation, so one of the hypotheses they require is automatically achieved in this context. Also since our groups have squarefree order, in their statement $e=1$, and in their notation $\Omega(2n)$ is the number of prime divisors of our $|\Omega|$, which is $s+1$.

\begin{thm}[Dobson, Muzychuk, Spiga,~\cite{DMS-CI3}]\label{blocks}
Let $R_r$ be a dihedral group of squarefree order acting regularly on the set\/ $\Omega$, and $R_r^\pi$ another such group. Then there exists $\beta \in \langle R_r, R_r^\pi\rangle$ such that the group $\langle R_r, R_r^{\pi \beta}\rangle$ has a sequence of normal $G$-invariant partitions $\mathcal B_0 \prec \mathcal B_1 \prec \cdots \prec\mathcal B_{s+1}$, where $\mathcal B_0=\Omega$ consists of singleton sets, and $\mathcal B_{s+1}$ consists of a single block. Additionally, $\mathcal B_s$ consists of $2$ blocks of cardinality $p_1 \cdots p_s$.
\end{thm}

Notice that the number of these properly nested $G$-invariant partitions forces the cardinality of the blocks of $\mathcal B_i$ to be a prime multiple of the cardinality of the blocks of $\mathcal B_{i-1}$ for each $1 \le i \le s+1$. After relabeling the primes if necessary, we may conclude that $\mathcal B_i$ consists of $2p_{i+1}\cdots p_s$ blocks of cardinality $p_1 \cdots p_i$.

Since each $\mathcal B_i$ consists of the orbits of a normal subgroup of $G$, it must consist of the orbits of the unique (normal) subgroup of $R_r$ that has order $p_1 \cdots p_i$, and also of the unique (normal) subgroup of $R_r^\pi$ that has order $p_1 \cdots p_i$. 

\begin{cor}\label{better-blocks}
Let $R_r$ be a dihedral group of squarefree order acting regularly on the set\/ $\Omega$, and $R_r^\pi$ another such group. Then there exists $\beta \in \langle R_r, R_r^\pi\rangle$ such that the group $\langle R_r, R_r^{\pi \beta}\rangle$ has a sequence of normal $\langle R_r, R_r^{\pi \beta}\rangle$-invariant partitions $\mathcal B_0 \prec \mathcal B_1 \prec \cdots \prec\mathcal B_{s+1}$, where $\mathcal B_0=\Omega$ consists of singleton sets, and $\mathcal B_{s+1}$ consists of a single block. Additionally, $\mathcal B_s$ consists of $2$ blocks of cardinality $p_1 \cdots p_s$.

Furthermore, we may choose $\beta$ so that for each $1\le i \le s$, if $\rho_i$ has order $p_i$ in $R_r$ and $\sigma_i$ has order $p_i$ in $R_r^\pi$, then for any fixed block $B$ of $\mathcal B_{i-1}$, there is some $k_B$ such that for every $j$, $B{(\sigma_i^\beta)^j}=B{(\rho_i^{k_B})^j}$. 
\end{cor}

\begin{proof}
The first paragraph of this statement is~\Cref{blocks}. 

Take $\beta_s$ to be the $\beta$ given by~\Cref{blocks}. Let $G_s$ be the subgroup of $\langle R_r, R_r^{\pi \beta_s}\rangle$ that fixes each block of $\mathcal B_s$ setwise. We will work by downward induction to define $\beta_{i-1}$ and $G_{i-1}$ from $\beta_i$ and $G_i$, where $i \in \{1, \ldots, s\}$, and $G_i$ fixes every block of $\mathcal B_i$ setwise. Then we will show that $\beta'=\beta_s\cdots\beta_0$ has the desired property.

With $\beta_i$ and $G_i$ defined, let $P_{1,i}$ and $P_{2,i}$ be Sylow $p_i$-subgroups of $G_{i}$ that contain $\rho_i$ and $\sigma_i^{\beta_s\cdots \beta_{i}}$ respectively. By Sylow's Theorems, there is some $\beta_{i-1} \in G_i$ such that $P_{2,i}^{\beta_{i-1}}=P_{1,i}$, so $\sigma_i^{\beta_s\cdots\beta_{i-1}} \in P_{1,i}$. Furthermore, since $\beta_{i-1}$ fixes every block of $\mathcal B_j$ for $i \le j \le s$, we have $\sigma_j^{\beta_s\cdots \beta_{i-1}}=\sigma_j^{\beta_s\cdots \beta_{j-1}}$ in its action on the blocks of $\mathcal B_{j-1}$.
Let $G_{i-1}$ be the subgroup of $\langle R_r, R_r^{\pi \beta_s\ldots \beta_{i-1}}\rangle$ that fixes every block of $\mathcal B_{i-1}$ setwise. 

When this has been completed, note that for any $j$, $\sigma_j^\beta$ has the same action as $\sigma_j^{\beta_s\cdots \beta_{j-1}}$ on the blocks of $\mathcal B_{j-1}$. 

For any $i \in \{1, \ldots, s\}$, and any fixed block $B$ of $\mathcal B_{i-1}$, there is some block of $\mathcal B_i$ that is the union of $\{B^{\rho_i^j}: 0 \le j \le p_i-1\}$. Now, $P_{1,i}$ is a $p_i$-group acting with degree $p_i$ on this set of $p_i$ blocks of $\mathcal B_{i-1}$, so must be acting as a cyclic group of order $p_i$. Since it contains $\langle \rho_i \rangle$, we must have $P_{1,i}=\langle \rho_i\rangle$ on this set. However, since $P_{1,i}$ also contains $\langle \sigma_i^{\beta'}\rangle$, we have $\langle \sigma_i^{\beta'}\rangle=P_{1,i}=\langle \rho_i\rangle$. Thus there is some $k_B$ such that  for every $j$, $B{(\sigma_i^{\beta'})^j}=B{(\rho_i^{k_B})^j}$. Replacing $\beta$ by $\beta'$ achieves the result.
\end{proof}

With this result, we are able to make some updates to our notation. The following notation includes~\Cref{notn-1} and more. To achieve the desired properties for the distinguished point $x$, we may replace each $\sigma_i$ by some power of itself if necessary.

\begin{notn}\label{notn-2}
Henceforth in this paper, we use the following notation:
\begin{itemize}
\item $p_1, p_2, \ldots, p_s$ are distinct primes;
\item $C_r=\langle \rho_1,\ldots, \rho_s\rangle$ is cyclic, with $|\rho_i|=p_i$ for each $1 \le i \le s$;
\item $C_r^\pi=\langle \sigma_1,\ldots,\sigma_s\rangle$ is cyclic, with $|\sigma_i|=p_i$ for each $1 \le i \le s$;
\item $R_r=\langle C_r,\tau_1\rangle$ where $|\tau_ 1|=2$ and $\alpha^{\tau_1}=\alpha^{-1}$ for every $\alpha \in C_r$;
\item $R_r^\pi=\langle C_r^\pi,\tau_2\rangle$ where $|\tau_ 2|=2$ and $\gamma^{\tau_2}=\gamma^{-1}$ for every $\gamma \in C_r^\pi$;
\item  both $R_r$ and $R_r^\pi$ are permutation groups acting regularly on the set $\Omega$ of cardinality $2p_1\cdots p_s$;
\item $x \in \Omega$ is a predetermined point;
\item $G=\langle R_r,R_r^\pi \rangle$;
\item $G$ admits invariant partitions $\mathcal B_0 \prec \cdots \prec \mathcal B_s$ such that $\mathcal B_0$ is the partition of $\Omega$ into singletons, and for each $1 \le i \le s$:
\begin{itemize}
\item $\mathcal B_i$ consists of $2p_{i+1}\cdots p_s$ blocks of cardinality $p_1\cdots p_i$;
\item $\mathcal B_i$ consists of the orbits of $\langle \rho_1, \ldots, \rho_i\rangle$, which are also the orbits of $\langle \sigma_1,\ldots, \sigma_i \rangle$;
\item  the block of $\mathcal B_{i-1}$ that contains $x{\sigma_i}$ is the same as the block of $\mathcal B_{i-1}$ that contains $x{\rho_i}$; and 
\item for any point $y \in \Omega$ lying in the block $B$ of $\mathcal B_i$, there is some $1 \le j_B\le p_i-1$ depending only on $B$, such that the block of $\mathcal B_{i-1}$ that contains $y{\sigma_i}$ is the same as the block of $\mathcal B_{i-1}$ that contains $y{\rho_i^{j_B}}$.
\end{itemize}
\item For every $y \in \Omega$ and every $1 \le i \le s$, we use $B_{i,y}$ to denote the block of $\mathcal B_i$ that contains the point $y$. If we have some other $G$-invariant partition denoted by some script letter, then we use a roman version of that letter with the subscript $y$ to denote the block of that partition that contains $y$. For example, in $\mathcal C$, we use $C_y$.
\end{itemize}
\end{notn}

In many situations, we may wish to work with a different ordering for the primes $p_1, \ldots p_s$. As long as all of the properties of~\Cref{notn-2} still hold, all of the results that follow still apply to this reordering. It will be important to our proofs to understand when we can do this; this is addressed in our next result. Essentially, this explains that whenever we have a $G$-invariant partition, we can pull the prime divisors of its block cardinalities to the front of our ordering, replacing some $\mathcal B_i$ by this $G$-invariant partition.

\begin{lem}\label{reordering}
Using~\Cref{notn-2}, let $\mathcal C$ be a $G$-invariant partition such that $\mathcal C \preceq \mathcal B_s$. Then there is a permutation $\varphi$ of $\{1,\ldots, s\}$ so that $G$ admits invariant partitions $\mathcal C_0 \prec \cdots \prec \mathcal C_s$ with the following properties:
\begin{itemize}
\item $\mathcal C_0=\mathcal B_0$ and $\mathcal C_s=\mathcal B_s$;
\item there is some $t$ such that $\mathcal C_t=\mathcal C$;
\item for each $1 \le i \le s$, $\mathcal C_i$ consists of $2p_{(i+1)\varphi}\cdots p_{s\varphi}$ blocks of cardinality $p_{1\varphi} \cdots p_{i\varphi}$;
\item for each $1 \le i \le s$, $\mathcal C_i$ consists of the orbits of $\langle \rho_{1\varphi},\ldots, \rho_{i\varphi}\rangle$, which are also the orbits of $\langle \sigma_{1\varphi},\ldots, \sigma_{i\varphi}\rangle$;
\item for each $1 \le i \le s$, the block of $\mathcal C_{i-1}$ that contains $x{\sigma_{i\varphi}}$ is the same as the block of $\mathcal C_{i-1}$ that contains $x{\rho_{i\varphi}}$; and
\item for each $1 \le i \le s$, for any point $y \in \Omega$ lying in the block $C$ of $\mathcal C_i$, there is some $1 \le j_C \le p_{i\varphi}-1$ depending only on $C$, such that the block of $\mathcal C_{i-1}$ that contains $y{\sigma_{i\varphi}}$ is the same as the block of $\mathcal C_{i-1}$ that contains $y{\rho_{i\varphi}^{j_C}}$.
\end{itemize}
In short, \Cref{notn-2} holds for this new ordering of our primes and this new corresponding collection of nested partitions.
\end{lem}

\begin{proof}
By~\Cref{block-form-not-regular}, since $\mathcal C$ is $R_r$-invariant, it consists of $\{x{H\gamma}: \gamma \in R_r\}$ for some $H \le R_r$. Since by hypothesis $\mathcal C \preceq \mathcal B_s$, we have $H \le C_r$ is cyclic and normal in $R_r$, and using~\Cref{block-form}, the orbits of $H$ are the blocks of $\mathcal C$. Likewise, since $\mathcal C\preceq \mathcal B_s$ is $R_r^\pi$-invariant, its blocks are the orbits of some cyclic $K \triangleleft R_r^\pi$.

Define $i_1, \ldots, i_t$ to be the values of $\{1, \ldots, s\}$, in ascending order, such that for $1 \le j \le t$, $\rho_{i_j}$ lies in $H$.
Now we define $\varphi$ as follows. For $1 \le j \le t$, define $j\varphi=i_j$. For $t<j \le s$, define $j\varphi$ to be the first value from $\{1, \ldots, s\}$ that does not appear in $\{1\varphi, \ldots, (j-1)\varphi\}$.

The first three points will follow immediately from the fourth together with the way we have chosen $i_1, \ldots, i_t$, so our first goal is to establish that for each $1\le j \le s$, if $\mathcal C_j$ consists of the orbits of $\langle \rho_{1\varphi},\ldots, \rho_{j\varphi}\rangle$, that these are also the orbits of $\langle \sigma_{1\varphi},\ldots, \sigma_{j\varphi}\rangle$, and that these partitions are $G$-invariant.

Suppose first that $j \le t$. Observe that $\mathcal B_{j\varphi}$ consists of the orbits of both $\langle \rho_1, \ldots, \rho_{j\varphi}\rangle$ and $\langle \sigma_1, \ldots, \sigma_{j\varphi}\rangle$, and $\mathcal C$ consists of the orbits of both $H$ and $K$. Therefore the intersection of $\langle \rho_1, \ldots, \rho_{j\varphi}\rangle$ with $H$ is a cyclic subgroup of odd order in $R_r$ that must have the same orbits as the intersection of $\langle \sigma_1, \ldots, \sigma_{j\varphi}\rangle$ with $K$, which is a cyclic subgroup of odd order in $R_r^\pi$. But these intersections are exactly $\langle \rho_{1\varphi},\ldots, \rho_{j\varphi}\rangle$ and $\langle \sigma_{1\varphi},\ldots, \sigma_{j\varphi}\rangle$. Thus the orbits of these two groups coincide, so by~\Cref{orbits-blocks} they form a $G$-invariant partition (which is $\mathcal C_j$).

Now suppose $j>t$. In this case, we apply~\Cref{block-union} to $\mathcal C$ and $\mathcal B_k$,  where $k$ is the $(j-t)$th value of $\{1, \ldots, s\}-\{i_1, \ldots, i_t\}$. The resulting $G$-invariant partition has blocks of cardinality $p_{1\varphi}\cdots p_{j\varphi}$ that are the orbits of $\langle \rho_{1\varphi},\ldots, \rho_{j\varphi}\rangle$ and also of $\langle \sigma_{1\varphi}, \ldots, \sigma_{j\varphi}\rangle$.  

This establishes the first four bullet points.

If we can establish the final bullet point, then if necessary we can replace each $\sigma_i$ by some power of itself to ensure that the other (penultimate) bullet point is also true, so we conclude our proof by establishing the final point. We know that the blocks of $\mathcal C_{i-1}$ are the orbits of $\langle \rho_{1^\varphi},\ldots, \rho_{(i-1)^\varphi}\rangle$ and of $\langle \sigma_{1^\varphi},\ldots, \sigma_{(i-1)^\varphi}\rangle$, and that the blocks of $\mathcal C_i$ are the orbits of $\langle \rho_{1^\varphi},\ldots, \rho_{i^\varphi}\rangle$ and of $\langle \sigma_{1^\varphi},\ldots, \sigma_{i^\varphi}\rangle$. Thus within any block $C$ of $\mathcal C_i$, there are $p_{i^\varphi}$ blocks of $\mathcal C_{i-1}$, and these are moved in a $p_{i^\varphi}$-cycle by both $\rho_{i^\varphi}$ and $\sigma_{i^\varphi}$. That these cycles lie in a single group of order $ p_{i^\varphi}$ is straightforward to show, using the structures of the blocks of $\mathcal C_{i-1}$ and $\mathcal C_i$ as described above, and the fact that $\rho_{i^\varphi}$ and $\sigma_{i^\varphi}$ lie in the same group of order $p_{i^\varphi}$ in their actions on the blocks of $\mathcal B_{(i-1)^\varphi}$ in any block of $\mathcal B_{i^\varphi}$.
\end{proof}

Understanding Cayley graphs requires an understanding of regular group actions, and we continue this section with a few notes on how this concept interacts with invariant partitions.

\begin{defn}
The action of the group $G$ is \emph{regular} in its action on the set $\Omega$ if for every pair of elements $y,z \in \Omega$, there is a unique $\gamma \in G$ such that $y\gamma=z$.
\end{defn}

When the action of $G$ is faithful and transitive, the following definition is equivalent:

\begin{defn}\label{def-regular}
The action of the faithful transitive group $G$ is \emph{regular} in its action on the set $\Omega$ if every element of $G$ that fixes a point of $\Omega$, fixes every point of $\Omega$.
\end{defn}

However, if we consider the action of a group of permutations of $\Omega$ on the set of blocks of some invariant partition $\mathcal B$, this action may not be faithful (there may be a nontrivial kernel; for example, if $\mathcal B=\mathcal B_1$, then $\rho_1$ and $\sigma_1$ are in the kernel). It may happen that every element of $G$ that fixes one block of $\mathcal B$ fixes every block of $\mathcal B$, but because the kernel of the action on $\mathcal B$ is nontrivial, if $B, B' \in \mathcal B$, $g_1 \in G$ with $Bg_1=B'$, and $g_2$ is a nontrivial element of $G$ that fixes every block of $\mathcal B$, then $Bg_1g_2=B'$. Thus there are multiple elements of $G$ that map $B$ to $B'$. To make this distinction, we use the following notation.

\begin{notn}
If $G$ acts transitively on the set $\Omega$, and $\mathcal B$ is a $G$-invariant partition of $\Omega$, then $G_{\mathcal B}$ denotes the group of permutations of the blocks of $\mathcal B$ induced by the action of $G$ on these blocks. 
\end{notn} 

These concepts lead to a definition that will be important to our understanding of the group actions in this paper.

\begin{defn}
Let $G$ be a permutation group acting transitively on the set $\Omega$, and let $\mathcal B$ be a $G$-invariant partition of $\Omega$. We say that $G$ is \emph{block-regular on $\mathcal B$} if every element of $G$ that fixes some $B \in \mathcal B$ fixes every $B' \in \mathcal B$.
\end{defn}

Thus, when we say that the action of $G$ is block-regular on $\mathcal B$, we mean that although $G_{\mathcal B}$ may have a nontrivial kernel (so that more than one element of $G$ maps one block to another), $G_\mathcal B$ would satisfy~\Cref{def-regular} if faithfulness were not required.

We will frequently be working with subgroups of $G$ that fix an element of $\Omega$, or that fix some subset of $\Omega$ (typically a block of a $G$-invariant partition) setwise. We use the standard notation $G_z$ for the subgroup of $G$ that fixes $z \in \Omega$. If $Y \subset \Omega$, then $G_Y$ denotes the setwise stabiliser of $Y$ in $G$.

\begin{lem}\label{blocks-in-F1}
Use~\Cref{notn-2}. Suppose that for some $1 \le i \le s-1$ the orbits of $\langle \rho_i\rangle$ form a $G$-invariant partition $\mathcal C$, and that there is some $\alpha \in C_r$ such that $G_{C_x}=G_{C_{x\tau_1\alpha}}$. Suppose also that for some $i < j \le s$, the orbits of $\langle \rho_i,\rho_j\rangle$ in $F_1$ are invariant under $\langle C_r, C_r^\pi\rangle$. Then the orbits of $\langle \rho_i,\rho_j\rangle$ are $G$-invariant.
\end{lem}

\begin{proof}
By~\Cref{orbits-blocks}, it is sufficient to show that the orbits of $\langle \rho_i, \rho_j\rangle$ coincide with the orbits of $\langle \sigma_i,\sigma_j\rangle$. Since the orbits of $\langle \rho_i,\rho_j\rangle$ in $F_1$ are invariant under $\langle C_r, C_r^\pi\rangle$, they do coincide with the orbits of $\langle \sigma_i,\sigma_j\rangle$ in $F_1$. Furthermore, since $\mathcal C$ is $G$-invariant, the orbits of $\sigma_i$ and $\rho_i$ coincide everywhere.

Let $z \in F_2$ be arbitrary, and consider $C_z\sigma_j$. We must show that $C_z\sigma_j=C_z\rho_j^k$ for some $k$. Choose $\alpha_1 \in C_r$ such that $C_z=C_{x\tau_1\alpha\alpha_1}$. Then conjugation by $\alpha_1$ gives $G_{C_{x\alpha_1}}=G_{C_{x\tau_1\alpha\alpha_1}}=G_{C_z}$. Since the orbits of $\langle \rho_i, \rho_j\rangle$ coincide with the orbits of $\langle \sigma_i,\sigma_j\rangle$ in $F_1$, there is some $k$ such that $C_{x\alpha_1}\sigma_j\rho_j^{-k}=C_{x\alpha_1}$, so $\sigma_j\rho_j^{-k} \in G_{C_{x\alpha_1}}=G_{C_z}$. Therefore $C_z\sigma_j=C_z\rho_j^k$, as desired, completing the proof.
\end{proof}

Very often in this paper we will be considering a group that induces an action on some set of prime cardinality $p$. It will be important to have a strong understanding of such actions.

\begin{lem}\label{Gy-affine}
Suppose that a permutation group $G$ fixes a set $D$ of prime cardinality $p$ (setwise).  

Then one of the following is true:
\begin{enumerate}
\item $G$ fixes every element of $D$;
\item $G$ acts transitively on the elements of $D$ and has an element of order $p$ that also acts transitively on $D$; or
\item there is some unique $d \in D$ such that for every $g \in G$, $dg=d$.\label{affine}
\end{enumerate}

Furthermore, if $G \le H$ and $H$ also fixes $D$ setwise, and $H$ is transitive on $D$, with $\rho \in H$ acting transitively on $D$ and $H$ not doubly-transitive on $D$, then in case~(\ref{affine}), $g$ normalises $\rho$.
\end{lem}

\begin{proof}
This is an easy consequence of a result by Burnside that every permutation group of prime degree is either doubly transitive, or affine. If such a group is not transitive, then, it normalises a cyclic group of order $p$ and the action of any non-identity element is as claimed. 
\end{proof}

The next result is also well-known, but important in this context. 

\begin{lem}\label{2-closure-invariant}
Let $G$ be a group acting transitively on the set $\Omega$ and let $\mathcal B$ be a $G$-invariant partition. Then $\mathcal B$ is also $G^{(2)}$-invariant.
\end{lem}

\begin{proof}
Let $B \in \mathcal B$ and $\beta \in G^{(2)}$. It is sufficient to observe that for any $u, v \in B$ there is some $g \in G$ such that $u\beta,v\beta \in Bg$. This is immediate from the definition of $2$-closure, since there is some $g \in G$ such that $(u\beta,v\beta)=(ug,vg)$.
\end{proof}

Note that whenever $\beta \in G^{(2)}$, we must have $\langle R_r, R_r^{\pi\beta}\rangle^{(2)} \le \langle R_r,R_r^\pi\rangle^{(2)}$. Since $G^{(2)}$ is a supergroup of $G$ and therefore by~\Cref{2-closure-invariant} admits exactly the same invariant partitions as $G$, this means that after conjugation any previously invariant partition remains invariant. In essence, if we have already conjugated some parts of $R_r^\pi$ to make this group closer to $R_r$, any further conjugation can't mess up things we've already straightened out. This is how we can justify the abuse of notation we noted at the beginning of this section.

\section{An equivalence relation and its equivalence classes}

In this section we are going to define an equivalence relation, prove that it is an equivalence relation, and deduce some properties of the $G$-invariant partition formed by its equivalence classes. We will end by introducing another partition that we will also use at times. We begin by defining the relation.

\begin{defn}\label{def-equivB}
Let $G$ be a transitive permutation group acting on a set $\Omega$, and
let $\mathcal B$ be a $G$-invariant partition with blocks of prime cardinality that are the orbits of some semiregular element $\rho$. For any point $y$ use $B_y$ to denote the block of $\mathcal B$ that contains $y$. 

We define the relation $\equiv_{\mathcal B}$ on the points of $\Omega$ by $y \equiv_{\mathcal B} z$ if there is a sequence of points $y_1=y, \ldots, y_k=z$ such that for each $1 \le i \le k-1$, $B_{y_{i+1}}$ is not contained in any orbit of $G_{y_i}$.
\end{defn}

In order to work with this relation, it is convenient to have a shorthand terminology for a sequence having the property we are looking for.

\begin{defn}
Suppose we have the relation $\equiv_{\mathcal B}$ defined in~\Cref{def-equivB}. If $y_1, \ldots, y_k$ is a sequence of points such that for each $1 \le i \le k-1$, $B_{y_{i+1}}$ is not contained in any orbit of $G_{y_i}$ then we say that $y_1, \ldots, y_k$ is an $\equiv_{\mathcal B}$-chain from $y_1$ to $y_k$.
\end{defn}

The following useful result has a similar flavour to results that have appeared previously in various forms such as Lemma 2 of \cite{Dobson1995}, but the details are rather different. For any point $y \in \Omega$, we use the standard notation $G_y$ to denote the subgroup of $G$ that fixes the point $y$.

\begin{lem}\label{X-blocks} 
Let $\equiv_{\mathcal B}$ be the relation defined in~\Cref{def-equivB}. Then $\equiv_{\mathcal B}$ is an equivalence relation on the points of $\Omega$, and consequently its equivalence classes form a $G$-invariant partition.
\end{lem}

\begin{proof}
For any $y$, clearly $B_y$ is not contained in any single orbit of $G_y$, so $y_1=y_2=y$ is an $\equiv_{\mathcal B}$ chain from $y$ to $y$. Thus $\equiv_{\mathcal B}$ is reflexive.

For any $y, z$, we now show that if $B_z$ is not contained in an orbit of $G_y$, then $B_y$ is not contained in an orbit of $G_z$. We will actually show the contrapositive, so suppose $B_y$ is contained in an orbit of $G_z$; this means that the subgroup of $G_z$ that fixes $B_y$ setwise, is transitive on $B_y$. Since $B_y$ has prime cardinality, by~\Cref{Gy-affine} there must be an element $\gamma \in G_z$ that has order $p$ and acts transitively on $B_y$. Since $\gamma$ has order $p$ all of its orbits have length $1$ or $p$, and since it fixes $z$ and $B_z$ has cardinality $p$, $\gamma$ must fix every point of $B_z$. For any $1 \le i \le p-1$ there is some $j$ such that $y{\rho^{-i}}=y{\gamma^j}$ (where $\rho$ is the element from~\Cref{def-equivB} whose orbits form the blocks of $\mathcal B$). Then $\gamma^j\rho^i \in G_y$, and $z{\gamma^j\rho^i}=z{\rho^i}$, so $B_z$ is contained in an orbit of $G_y$, as claimed. This implies that if $y_1, \ldots, y_k$ is an $\equiv_{\mathcal B}$-chain from $y$ to $z$, then $y_k, \ldots, y_1$ is an $\equiv_{\mathcal B}$-chain from $z$ to $y$, so the relation $\equiv_{\mathcal B}$ is symmetric.

Finally, if $y \equiv_{\mathcal B} z$ and $z \equiv_{\mathcal B} u$ then there is an $\equiv_{\mathcal B}$ chain $y_1, \ldots, y_k$ from $y$ to $z$ and an $\equiv_{\mathcal B}$-chain $z_1, \ldots, z_\ell$ from $z$ to $u$. Concatenating gives an $\equiv_{\mathcal B}$-chain $y_1, \ldots, y_k=z_1, \ldots, z_\ell$ from $y$ to $u$. Thus $\equiv_{\mathcal B}$ is transitive.

We have shown that this is an equivalence relation; since $\mathcal B$ is $G$-invariant it is easy to see that the equivalence classes must be $G$-invariant.
\end{proof}

The following result is a key concept that we will use often in this paper; we will need it for the first time to prove one of the important properties we'll need to know about our partition.

\begin{lem}\label{pth-power}
Use~\Cref{notn-2}. Suppose we know that for some $y \in \Omega$, some $i$, some $\alpha\in C_r$ whose order is not divisible by $p_i$, some $1 \le t<p_i$, and whenever $z=y\sigma_i^j$ for some $j \ge 0$, we have $$z{\sigma_i}=z{\rho_i^t\alpha}.$$ Then $\alpha$ is the identity, so for every $j$, $$y{\sigma_i^j}=y{\rho_i^{tj}}.$$
\end{lem}

\begin{proof}
Note that by applying our hypothesis $p_i$ times, we obtain $y\sigma_i^{p_i}=y(\rho_i^t\alpha)^{p_i}$.
Since $\sigma_i$ has order $p_i$ we have $y{\sigma_i^{p_i}}=y$. Since $\rho_i$ and $\alpha$ commute and $\rho_i$ has order $p_i$, $y{(\rho_i^t\alpha)^{p_i}}=y{\alpha^{p_i}}$. This implies that $y{\alpha^{p_i}}=y$, but since $p_i$ does not divide the order of $\alpha$, the orbit-stabiliser theorem implies that no orbit of $\langle \alpha\rangle$ can have length $p_i$. Since $p_i$ is prime, the only way the equation $y{\alpha^{p_i}}=y$ can be satisfied is if $\alpha$ is the identity. That $y{\sigma_i^j}=y{\rho_i^{tj}}$ for every $j$ follows immediately.
\end{proof}

The next lemma establishes some useful properties of the $G$-invariant partitions we have just produced. 

\begin{lem}\label{sigmas-on-X} Use~\Cref{notn-2}.
Let $\mathcal X$ be the $G$-invariant partition arising from the equivalence classes of $\equiv_{\mathcal B}$ (note this requires that the blocks of $\mathcal B$ have prime cardinality). Then the following hold:
\begin{enumerate}
\item $\mathcal X \succeq \mathcal B$.\label{B-in-X}
\item Suppose that the orbits of $\langle \rho_i,\rho_j\rangle$ form a $G$-invariant partition $\mathcal C $ with $\mathcal B \preceq \mathcal C \preceq \mathcal X$, and that $\sigma_j$ commutes with $\rho_i$. Then there is a constant $k_C$ such that $\sigma_j=\rho_j^{k_C}$ on any point in $C$.

Furthermore, for every $X \in \mathcal X$ there is a constant $k_X$ such that $\sigma_i=\rho_i^{k_X}$ on any point in $X$.\label{X-constant}
\item If the orbits of $\langle \rho_j\rangle$ form a $G$-invariant partition $\mathcal C$, then $\mathcal X \succeq \mathcal C$.\label{C-in-X}
\end{enumerate}
\end{lem}

\begin{proof}
\begin{enumerate}
\item If $y, z$ are in the same block of $\mathcal B$ then $B_y=B_z$. Since $B_y$ is not contained in an orbit of $G_y$, we conclude that $y \equiv_{\mathcal B} z$. Thus $\mathcal X \succeq \mathcal B$.

\item Fix $C \in \mathcal C$, or if $i=j$ then take $C \in \mathcal X$. Since the orbits of $\langle \rho_i,\rho_j\rangle$ are $G$-invariant (whether or not $i=j$), using~\Cref{notn-2} there must be some $k_C$ such that $\sigma_j\rho_j^{-k_C}$ fixes every block of $\mathcal B$ in $C$. We will show that $\sigma_j\rho_j^{-k_C}$ fixes every point in $C$.

 Let $y, z \in C \subseteq X \in \mathcal X$, and let $\ell$ be such that $y{\sigma_j}\rho_j^{-k_C}=y{\rho_i^\ell}$. By definition of $\equiv_{\mathcal B}$, there is an $\equiv_{\mathcal B}$-chain $y_1, \ldots, y_b$ from $y$ to $z$. Let $1 \le a \le b-1$, and suppose that $y_{a}{\sigma_j\rho_j^{-k_C}}=y_a{\rho_i^{\ell}}$ (this is true for $a=1$). Then $\sigma_j\rho_j^{-k_C}\rho_i^{-\ell} \in G_{y_a}$, and $B_{y_{a+1}}$ is not contained in an orbit of $G_{y_a}$. 

By hypothesis $\sigma_j$ commutes with $\rho_i$, so for every $d$ we have $$y_{a+1}\rho_i^d(\sigma_j\rho_j^{-k_C}\rho_i^{-\ell})=y_{a+1}(\sigma_j\rho_j^{-k_C}\rho_i^{-\ell})\rho_i^d.$$ This means that if $$y_{a+1}\sigma_j\rho_j^{-k_C}\rho_i^{-\ell}=y_{a+1}\rho_i^c$$ then $$y_{a+1}\rho_i^d(\sigma_j\rho_j^{-k_C}\rho_i^{-\ell})=y_{a+1}\rho_i^d\rho_i^c;$$ that is, $\sigma_j \rho_j^{-k_C}\rho_i^{-\ell}$ acts as $\rho_i^c$ on $B_{y_{a+1}}$. Since $B_{y_{a+1}}$ is not contained in an orbit of $G_{y_{a}}$, we must have $c=0$, so $$y_{a+1}\sigma_j\rho_j^{-k_C}=y_{a+1}\rho_i^{\ell}.$$

Inductively, we see that for every $a$ we have $y_{a+1}{\sigma_j\rho_j^{-k_C}}=y_{a+1}{\rho_i^{\ell}}$. In particular, $z{\sigma_j\rho_j^{-k_C}}=z{\rho_i^{\ell}}$. Since $z$ was an arbitrary element of $C$ and $C$ is a union of orbits of $\langle \sigma_j\rangle$, in particular this is true whenever $z=y\sigma_j^m$ for some $m$. Thus by~\Cref{pth-power}, $\rho_i^\ell$ is the identity, so $y\sigma_j=y\rho_j^{k_C}$. Since $y$ was arbitrary, $\sigma_j=\rho_j^{k_C}$ on any point in $C$.

\item Suppose $y \in \Omega$ and $z \in C_y$. After using (\ref{B-in-X}), we may assume that $\mathcal C \neq \mathcal B$. Then $y$ is the unique element of $C_y \cap B_y$, and $z$ is the unique element of $C_y \cap B_z$. Since every element of $G_y$ must fix $C_y$ setwise, any element of $G_y$ that fixes $B_z$ setwise must also fix $z$, so $B_z$ does not lie in a single orbit of $G_y$. Therefore $y$ and $z$ lie in the same block of $\mathcal X$.

\end{enumerate}
\end{proof}

There is another more straightforward equivalence relation whose equivalence classes produce a $G$-invariant partition. This is little more than an observation that has been made by many others.

\begin{lem}\label{Gx-blocks}
Let $G$ be a group acting transitively on the set $\Omega$. Define an equivalence relation $R$ on $\Omega$ by given $x, y \in \Omega$, $x R y$ iff $G_x=G_y$. Then the equivalence classes of $R$ form a $G$-invariant partition.

In fact, if~\Cref{notn-2} applies, $G_x=G_y$ and $\alpha \in C_r$ with $x\alpha=y$, then the orbits of $\alpha$ are $G$-invariant.
\end{lem}

\begin{proof}
That $R$ is an equivalence relation is clear since the relation is defined based on equality. Let $g \in G$ and $x, y \in \Omega$ with $xRy$. Then $h \in G_{xg}$ if and only if $ghg^{-1}\in G_x$, which is true if and only if $ghg^{-1} \in G_y$, which is true if and only if $h \in G_{yg}$. So $xgRyg$. This proves the first paragraph.

Suppose $x\alpha=y$ with $\alpha \in C_r$ and $G_x=G_y$. Take two arbitrary elements in the same $\alpha$-orbit, say $z$ and $z\alpha^i$ for some $i$, and any $g \in G$. We will show that $(z\alpha^i)g=(zg)\alpha^{\pm i}$, so that $zg$ and $z\alpha^ig$ are in the same $\alpha$-orbit.

Conjugating $G_x$ by $\alpha$ gives $G_x=G_y=G_{x\alpha}=G_{y\alpha}=G_{x\alpha^2}$. Continuing inductively, $G_x=G_{x\alpha^j}$ for every $j$. In particular, $G_x=G_{x\alpha^i}$. Let $h \in G$ such that $xh=z$. Then conjugating by $h$ gives $G_z=G_{z\alpha^i}$. Let $\alpha_1\in R_r$ such that $zg=z\alpha_1$, so $g\alpha_1^{-1} \in G_z=G_{z\alpha^i}$. Then $z\alpha^ig\alpha_1^{-1}=z\alpha^i$, so $$z\alpha^ig=z\alpha^i\alpha_1=z\alpha_1\alpha^{\pm i}=zg\alpha^{\pm i},$$ as desired.
\end{proof}

Frequently when applying the relation above, we will be considering the action $G_{\mathcal B}$ of $G$ on the blocks of some $G$-invariant partition $\mathcal B$. Although this technically defines a relation on the blocks of $\mathcal B$, again we often abuse notation by identifying this relation with the relation it induces on the elements of $\Omega$, defined by $xRy$ iff $B_xRB_y$.

\section{More preliminaries}

In this section we prove some additional preliminary results that we will need to use in our main proofs.

It is important to be aware that if we can find some $\beta\in G^{(2)}$ such that $C_r^{\pi\beta}=C_r$, then $R_r^{\pi\beta}=R_r$, since there is a unique regular dihedral group containing any semiregular cyclic group of index $2$. We show this in the following proposition.

\begin{prop}\label{cyclic-enough}
Suppose that $R_1$ and $R_2$ are regular dihedral permutation groups acting on a set $\Omega$, whose index-$2$ semiregular cyclic subgroups $C_1$ and $C_2$ are equal. Then $R_1=R_2$.
\end{prop}

\begin{proof}
Let $\sigma$ generate $C_1=C_2$ acting semiregularly with two orbits on $\Omega$. Let $\tau\in R_1-C_1$, and $\tau' \in R_2-C_1$ (so $\tau$ and $\tau'$ are reflections in the two groups). We will show that $\tau' \in R_1$, which is sufficient.

Let $x \in \Omega$. Note that the orbits of $C_1$ partition $\Omega$ into two sets: $x{C_1}$, and $x{\tau C_1}$. Notice also that we must have $x {\tau'}\in x{\tau C_1}$. Let $j$ be such that $x{\tau'}=x{\tau\sigma^j}$. Then for any $k$, $$(x{\sigma^k}){\tau'}=x{\tau'\sigma^{-k}}=x{\tau\sigma^{j-k}}=(x{\sigma^k}){\tau\sigma^j}.$$ So $\tau'$ has the same action as $\tau\sigma^j$ on every element in the orbit of $x$ under $C_1$.

For any $z \in \Omega$ that is not in the orbit of $x$ under $C_1$, we have $z=y{\tau'}=y{\tau\sigma^j}$ for some $y$ that is in the orbit of $x$ under $C_1$. Since $\tau'$ and $\tau\sigma^j$ are involutions, $y=z{\tau'}=z{\tau\sigma^j}$. So $\tau'$ has the same action as $\tau\sigma^j$ on every element of $\Omega$. Hence $\tau'=\tau\sigma^j \in R_1=\langle \tau,\sigma\rangle$.
\end{proof}

Since we may at times choose an initial $\beta$ that conjugates $\sigma_i$ to $\rho_i$ for a specific $i$, it is helpful to know what we can deduce about how $\rho_i$ interacts with other elements of $G$ once we know that $\sigma_i=\rho_i$.

\begin{lem}\label{commuting}
Use~\Cref{notn-2}. Suppose $\sigma_i=\rho_i$ and $g \in G$. Then $F_1g=F_1$ if and only if $g$ commutes with $\rho_i$, while $F_1g=F_2$ if and only if $g$ inverts $\rho_i$.
\end{lem}

\begin{proof}
We know that every element of $C_r$ commutes with $\rho_i$, and $\tau_1$ inverts $\rho_i$. Also, since $\rho_i=\sigma_i$, every element of $C_r^\pi$ commutes with $\rho_i$, and $\tau_2$ inverts $\rho_i$. Since $g \in G=\langle R_r, R_r^\pi\rangle$, we can write $g$ as a word in $\rho_1,\sigma_1, \ldots, \rho_s,\sigma_s, \tau_1,\tau_2$. 

With any such representation of $g$, it is not hard to see that $\rho_i$ commutes with $g$ if the total number of appearances of $\tau_1$ and $\tau_2$ is even, and $\rho_i$ is inverted by $g$ if the total number of appearances of $\tau_1$ and $\tau_2$ is odd. Since every element of $C_r$ and $C_r^\pi$ fixes $F_1$, while $\tau_1$ and $\tau_2$ exchange $F_1$ with $F_2$, we also have $F_1g=F_1$ if and only if the total number of appearances of $\tau_1$ and $\tau_2$ is even, which happens if and only if $g$ commutes with $\rho_i$. Similarly the total number of appearances of $\tau_1$ and $\tau_2$ is odd if and only if $g$ inverts $\rho_i$ and equivalently $F_1g=F_2$.
\end{proof}

Recalling that one condition of~\Cref{sigmas-on-X}(\ref{X-constant}) was that $\sigma_j$ commutes with $\rho_i$, the following result provides conditions under which this is true.

\begin{lem}\label{commuting-refined}
Use~\Cref{notn-2}. Let $\mathcal B$ be a $G$-invariant partition and let $i, j$ be such that $\sigma_i, \rho_i, \sigma_j$ and $\rho_j$ fix every block of $\mathcal B$. Suppose that for each $B \in \mathcal B$, there is some $k_B$ with $1 \le k_B \le p_i-1$ such that $\sigma_i=\rho_i^{k_B}$ on every point of $B$.

Then $\sigma_j$ commutes with $\rho_i$.
\end{lem}

\begin{proof}
Let $y \in \Omega$. Then $y\sigma_j \in B_y$, so if we let $k^{-1}_{B_y}$ be the multiplicative inverse of $k_{B_y}$ in $\mathbb Z_{p_i}$ we have $$y\sigma_j\rho_i=y\sigma_j \sigma_i^{k^{-1}_{B_y}}=y\sigma_i^{k^{-1}_{B_y}}\sigma_j=y\rho_i^{k_{B_y}k^{-1}_{B_y}}\sigma_j=y\rho_i\sigma_j.$$
\end{proof}

Since our goal is to find $\beta\in G^{(2)}$ such that $R_r^{\pi\beta}=R_r$, we will frequently need to prove that a particular permutation we define does indeed lie in the $2$-closure of $G$. Our next lemma will allow us to do this without excessive repetition of calculations.

\begin{lem}\label{in-2-closure}
Use~\Cref{notn-2}. Let $\beta$ be a permutation on $\Omega$ that fixes $F_1$ and $F_2$ setwise, and let $u, v \in \Omega$.
Suppose that exists a $G$-invariant partition $\mathcal D$ such that:
\begin{itemize} 
\item there is some $g \in G$ such that $(D_u,D_v)\beta=(D_u,D_v){g}$; and
\item $D_v$ lies in an orbit of $G_u$.
\end{itemize}
Then there is some $h \in G$ such that $(u,v)\beta=(u,v)h$.
\end{lem}

\begin{proof}
Note that the intersections of blocks of $\mathcal B_s$ with blocks of $\mathcal D$ forms a $G$-invariant partition, and by~\Cref{block-form}, its blocks are orbits of some normal subgroup of $C_r$.
We have $ug \in D_u\beta$; since $\beta$ fixes $F_1$ and $F_2$ setwise, 
there is some $\alpha \in C_r$ such that $ug\alpha=u\beta$ and $\alpha$ fixes every block of $\mathcal D$ setwise. 
Now since $\alpha$ fixes every block of $\mathcal D$ and $D_v\beta=D_vg$, it follows that $v\beta\alpha^{-1}g^{-1} \in D_v$. Since $D_v$ lies in an orbit of $G_u$, there is some $g_1 \in G_u$ such that $vg_1=v\beta\alpha^{-1}g^{-1}$, so $vg_1g\alpha=v\beta$. We also have $ug_1g\alpha=ug\alpha=u\beta$. Taking $h=g_1g\alpha$ yields the desired conclusion.
\end{proof}

In several circumstances, we will choose $\beta$ to have the following action, and of course we want to know what conjugation by $\beta$ does to various elements of $R_r^\pi$.

\begin{lem}\label{conj-works-2}
Use~\Cref{notn-2}. Fix $y \in \Omega$ and suppose that for some fixed $i,k$ and for every $j$, $\beta$ acts on $y{\sigma_i^j}$ as $\sigma_i^{-j}\rho_i^{kj}$. Then whenever $z=y{\rho_i^{kj}}$ for some $j$, we have $z{\sigma_i^\beta}=z{\rho_i^k}$.
\end{lem}

\begin{proof}
Let $\ell$ be such that $z=y{\rho_i^{k\ell}}$. Then we have $$z{\sigma_i^\beta}=z{\beta^{-1}\sigma_i\beta}=y{\rho_i^{k\ell}\beta^{-1}\sigma_i\beta}=y{\rho_i^{k\ell}\rho_i^{-k\ell}\sigma_i^\ell\sigma_i\beta}=y{\sigma_i^{\ell+1}\beta}$$
$$=y{\sigma_i^{\ell+1}\sigma_i^{-(\ell+1)}\rho_i^{k(\ell+1)}}=y{\rho_i^{k\ell}\rho_i^k}=z{\rho_i^k}.$$
\end{proof}

We conclude our preliminaries by describing one situation in which we may complete the proof immediately.

\begin{prop}\label{use-cyclic}
Use~\Cref{notn-2}. Suppose that $F_2$ is an orbit of $G_x$.
Then there is some $\beta \in G^{(2)}$ such that $R_r^{\pi\beta}=R_r$.
\end{prop}

\begin{proof}
Note that the restrictions of $C_r$ and $C_r^\pi$ to $F_1$ are regular cyclic groups of squarefree order. By~\Cref{Muz-cyclic} cyclic groups of this order have the DCI property. Thus by~\Cref{Babai-conj}, there is some $\beta_1 \in \langle C_r, C_r^\pi\rangle$ such that $C_r^{\pi\beta_1}=C_r$, where we are considering only the action of these groups on $F_1$. We can similarly find a $\beta_2$ in the 2-closure of the restriction of these actions to $F_2$ such that $C_r^{\pi\beta_2}=C_r$ on $F_2$. Extend $\beta_1$ and $\beta_2$ to permutations on $\Omega$ by having $\beta_1$ fix every point of $F_2$ and $\beta_2$ fix every point of $F_1$. 

We claim that $\beta=\beta_1\beta_2 \in G^{(2)}$. By our choices of $\beta_1$ and $\beta_2$, if $y,z \in F_i$ with $i \in \{1,2\}$ then it is immediate that there is some $g_i \in \langle C_r, C_r^{\pi}\rangle$ such that $y{\beta}=y{\beta_i}=y{g_i}$ and $z{\beta}=z{\beta_i}=z{g_i}$.  If $y \in F_1$ and $z \in F_2$, then taking $\mathcal D=\mathcal B_s$ in~\Cref{in-2-closure} gives some $g \in G$ such that $(y,z)\beta=(y,z)g$. Thus $\beta \in G^{(2)}$.

Finally to complete the proof we require $R_r^{\pi\beta\beta_3}=R_r$ for some $\beta_3 \in G^{(2)}$. If $y \in F_1$ and $\gamma \in C_r^\pi$ is a generator for $C_r^\pi$, then $y{\gamma^\beta}=y{\gamma^{\beta_1}}=y{\alpha_1}$ for some generator $\alpha_1$ for $C_r$, by the choice of $\beta_1$. Likewise, if $y \in F_2$ then $y{\gamma^\beta}=y{\gamma^{\beta_2}}=y{\alpha_2}$ for some generator $\alpha_2$ for $C_r$, by the choice of $\beta_2$. Unfortunately, it may be the case that $\alpha_2=\alpha_1^k$ for some $k \neq 1$. If this occurs, then we conjugate again by the map $\beta_3$ which acts as the identity on $F_1$, and as $\tau_2\tau_1$ on $F_2$. If $y \in F_2$ then $y{\gamma^{\beta\beta_3}}=y{\tau_1\tau_2\gamma^\beta\tau_2\tau_1}=y{\tau_1(\gamma^{\beta})^{-1}\tau_1}$. Since $y{\tau_1} \in F_1$, $\gamma^{\beta}$ has the same action as $\alpha_1$ on it, so this is $y{\tau_1\alpha^{-1}\tau_1}=y\alpha$. The same reasoning we used above to show that $\beta=\beta_1\beta_2 \in G^{(2)}$ shows $\beta_3 \in G^{(2)}$. 

We now have $C_r^{\pi\beta\beta_3}=C_r$, and by~\Cref{cyclic-enough}, this implies $R_r^{\pi\beta\beta_3}=R_r$.
\end{proof}

In the remaining sections, we will deal one at a time with the possibilities that $G$ is block-regular on $\mathcal B_1$, or $G$ is block-regular on $\mathcal B_2$, or $G$ is block-regular on $\mathcal B_3$. For the second and third of these, we will need to assume $s=3$. Note that when $s \le 3$, the third of these must always be true ($F_1$ is fixed setwise if and only if $F_2$ is fixed setwise).

\section{$G$ is block-regular on $\mathcal B_1$}

In this section we address the possibility that $G$ is block-regular on $\mathcal B_1$. Since this is the strongest of our possible hypotheses about the block-regularity of $G$, it is the only situation in which we are able to complete the conjugation for every value of $s$.

We will be using some additional notation repeatedly from this point, so we introduce it here although much of it will not be required until the next section.

\begin{notn}\label{notn-blocks}
The following partitions will arise in many of our proofs. See~\Cref{X-blocks} and~\Cref{Gx-blocks} in which it was proved that these partitions are  $G$-invariant.
\begin{itemize}
\item We use $\mathcal X$ to denote the $G$-invariant partition consisting of the equivalence classes of $\equiv_{\mathcal B_1}$; and
\item $\mathcal K=\{\{y \in \Omega: G_{B_{1,y}}=G_{B_{1,z}}\}: z \in \Omega\}$.
\end{itemize}
In addition, when there is a $G$-invariant partition $\mathcal C$ with blocks of cardinality $p_2$,
\begin{itemize}
\item  we use $\mathcal Y$ to denote the $G$-invariant partition consisting of the equivalence classes of $\equiv_{\mathcal C}$; and
\item $\mathcal L=\{\{y \in \Omega: G_{C_{y}}=G_{C_{z}}\}: z \in \Omega\}$.
\end{itemize}
\end{notn}

In our first result, we show that we can always conjugate $\sigma_1$ to $\rho_1$ in this situation.

\begin{lem}\label{reg-on-B1}
Use~\Cref{notn-2}. Suppose that $G$ is block-regular on $\mathcal B_1$. Then there is some $\beta \in G^{(2)}$ such that $\sigma_1^\beta=\rho_1$.
\end{lem}

\begin{proof}
Use $\mathcal X$ from~\Cref{notn-blocks} also.  For each block $X \in \mathcal X$, choose some representative point $y$, with $x$ being one of these representatives, choosing $y \in F_1$ if possible. For each representative $y$, define $\alpha_y$, $\gamma_y$ to be the unique elements of $R_r$ and $R_r^\pi$ (respectively) such that $x{\alpha_y}=x{\gamma_y}=y$. For every $z \in X_y$, define $z\beta=z{\gamma_y^{-1}\alpha_y}$. Note that since for every representative $y$ we have $y\beta=y$ and $G$ is block-regular on $\mathcal B_1$, $\beta$ fixes every block of $\mathcal B_1$ setwise.

We claim that $\beta \in G^{(2)}$, and that $\sigma_1^{\beta}=\rho_1$. 

Suppose $u, v \in \Omega$. If $u,v \in X_y$ then taking $g=\gamma_{y}^{-1}\alpha_{y}$ gives an element $g \in G$ such that $(u,v)\beta=(u,v)g$. If  $u$ and $v$ are in different blocks of $\mathcal X$, then $B_v$ lies in an orbit of $G_u$. Taking $\mathcal D=\mathcal B_1$ in~\Cref{in-2-closure} gives $h \in G$ such that $(u,v)\beta=(u,v)h$. Thus $\beta \in G^{(2)}$.

For any $u \in X_y$, since $\mathcal B_1 \preceq \mathcal X$ by~\Cref{sigmas-on-X}(\ref{B-in-X}), we have $$u{\sigma_1^\beta}=u{\beta^{-1}\sigma_1\beta}=u{\alpha_{y}^{-1}\gamma_{y}\sigma_1\gamma_{y}^{-1}\alpha_{y}}=u{\alpha_{y}^{-1}\sigma_1\alpha_{y}}.$$ We have $X_y{\alpha_{y}^{-1}}=X_x$, so $u{\alpha_y^{-1}}=v$ for some $v \in X_x$. Noting that $\sigma_1$ and $\rho_1$ have identical actions on $X_x$ (using~\Cref{sigmas-on-X}(\ref{X-constant})), this gives
$$u{\sigma_1^\beta}=v{\sigma_1\alpha_{y}}=v{\rho_1\alpha_{y}}=v{\alpha_y\rho_1}=u{\rho_1}.$$
Since $u$ was arbitrary, this completes the proof that $\sigma_1^\beta=\rho_1$.
\end{proof}

With this in hand, we can use one argument to conjugate any of the remaining generators of $C_r^\pi$.

\begin{lem}\label{reg-on-B1-done}
Use~\Cref{notn-2}. Fix $i \in \{2, \ldots, s\}$. Suppose that $G$ is block-regular on $\mathcal B_1$ and that $\sigma_m=\rho_m$ for each $1 \le m <i$. Then there is some $\beta_i \in G^{(2)}$ such that $\sigma_m^{\beta_i}=\sigma_m$ for each $1 \le m \le i$.
\end{lem}

\begin{proof}
Using~\Cref{notn-2}, we know that $x\sigma_i=x\rho_i\alpha_i$ for some $\alpha_i \in \langle \rho_j : 1 \le j \le i-1\rangle$. Since $G$ is block-regular on $\mathcal B_1$, this means that $\sigma_i\rho_i^{-1}\alpha_i^{-1}$ fixes every block of $\mathcal B_1$, so for every $B \in \mathcal B_1$, $B\sigma_i=B\rho_i\alpha_i$. Applying~\Cref{pth-power} to $G_{\mathcal B_1}$, we must have $\alpha_i$ in the kernel of $G_{\mathcal B_1}$, so $\sigma_i\rho_i^{-1}$ fixes every block of $\mathcal B_1$. 
Also, $\sigma_1=\rho_1$ is centralised by $\sigma_i$.

Observe that the orbits of $\langle \rho_1, \rho_i\rangle$ form a $G$-invariant partition $\mathcal C_i$. This follows from~\Cref{orbits-blocks} because as we have just observed, in $G_{\mathcal B_1}$, $\sigma_i$ and $\rho_i$ have the same action, so their orbits coincide.

Use $\mathcal X$ from~\Cref{notn-blocks}. Since $\sigma_m=\rho_m$ commutes with $\sigma_1=\rho_1$ for every $1 \le m <i$, we conclude using~\Cref{sigmas-on-X}(\ref{C-in-X}) that $\mathcal X \succeq \mathcal B_{i-1}$.

If $\mathcal X \succeq \mathcal C_i$, then~\Cref{sigmas-on-X}(\ref{X-constant}) tells us that for every $C \in \mathcal C_i$ there is a constant $k_{C}$ such that $\sigma_i=\rho_i^{k_C}$ on any point of $C$. Since $\sigma_i\rho_i^{-1}$ fixes every block of $\mathcal B_1$, we must have $k_C=1$ for every $C \in \mathcal C_i$, and thus $\sigma_i=\rho_i$ already. So we may assume that $\mathcal X \not\succeq \mathcal C_i$, from which it is straightforward to deduce that for each $X \in \mathcal X$, $X\sigma_i\neq X$.

For each block $X \in \mathcal X$, choose a representative point $y \in X$, with $x$ being one of these representatives; if possible, choose $y \in F_1$. Let $\gamma_y \in R_r^\pi$ and $\alpha_y \in R_r$ be such that $x\alpha=x\gamma=y$. 
 
For each $y \in \Omega$, define $Y_y=\{X_y\sigma_i^j: 0 \le j \le p_i-1\}$. For each $Y_y$ choose a representative point $z_y$, with $x=z_x$ and $z_y \in F_1$ whenever possible. 

Define $\beta_i$ as follows. Let $z_y$ be a representative for $Y_y$. If $z \in X_{z_y\sigma_i^j}$ with $0 \le j \le p_i-1$, then $z \beta_i=z \sigma_i^{-j}\rho_i^j$. 

We show first that $\beta_i \in G^{(2)}$. Let $u, v \in \Omega$. If $v \in X_u$ then we have $(u,v)\beta_i=(u,v)\sigma_i^{-j}\rho_i^j$ for some fixed $j$. If $v \notin X_u$ then $B_{1,v}$ lies in an orbit of $G_u$. Note that $\beta_i$ fixes every block of $\mathcal B_1$ setwise. Thus~\Cref{in-2-closure} produces some $h \in G$ such that $(u,v)\beta_i=(u,v)h$. Thus $\beta_i \in G^{(2)}$.

Now  since $\mathcal X \succeq \mathcal B_{i-1}$ and on any block of $\mathcal X$ we have $\beta_i=\sigma_i^{-j}\rho_i^j$ for some fixed $j$, which commutes with $\sigma_m=\rho_m$ whenever $1 \le m \le i-1$, we have $\sigma_m^{\beta_i}=\sigma_m$.  
 Also, applying~\Cref{conj-works-2} with any choice of $y$ and with $k=1$, we see that $\sigma_i^{\beta_i}=\rho_i$. This completes the proof.
\end{proof}

We tie the results from this section together into one corollary to make it easier to use later.

\begin{cor}\label{cor-reg-B1}
Use~\Cref{notn-2}. Suppose that $G$ is block-regular on $\mathcal B_1$. Then there is some $\beta \in G^{(2)}$ such that $R_r^{\pi\beta}=R_r$.
\end{cor}

\begin{proof}
\Cref{reg-on-B1} shows that after conjugation by some element $\beta_1$ of $G^{(2)}$, we have $R_r^{\pi\beta_1}$ has the element $\sigma_1^{\beta_1}=\rho_1$. We proceed to use~\Cref{reg-on-B1-done} inductively, to show that once we have $\sigma_i^{\beta_1\cdots\beta_k}=\rho_i$ for every $1 \le i \le k<s$, there exists $\beta_{k+1} \in G^{(2)}$ such that $\sigma_{k+1}^{\beta_1\cdots \beta_{k+1}}=\rho_{k+1}$ and $\sigma_i^{\beta_1\cdots\beta_{k+1}}=\rho_i$ for every $1 \le i \le k$.

Finally, taking $\beta=\beta_1\cdots \beta_s$, we arrive at $C_r^{\pi\beta}=C_r$, and so by~\Cref{cyclic-enough}, $R_r^{\pi\beta}=R_r$. 
\end{proof}

\section{$G$ is block-regular on $\mathcal B_2$}

In this section, we consider what happens if  $G$ is block-regular on $\mathcal B_2$. 

We begin with a result that is not specific to this section, but that we did not previously require.

\begin{lem}\label{K-in-X}
Use~\Cref{notn-2} and~\Cref{notn-blocks} and suppose that the orbits of $\langle \rho_2 \rangle$ are $G$-invariant so that $\mathcal Y$ and $\mathcal L$ are defined. Then $\mathcal Y \succeq \mathcal K$, and $\mathcal X \succeq \mathcal L$.
\end{lem}

\begin{proof}
Suppose $y$ and $z$ are in the same block of $\mathcal K$ so that $G_{B_{1,y}}=G_{B_{1,z}}$. Then $G_y$ fixes $B_{1,z}$ setwise so $C_z$ cannot lie in a single orbit of $G_y$. The other proof is similar.
\end{proof}

We first show that whenever $\mathcal X \succeq \mathcal B_2$, we have a second $G$-invariant partition with blocks of prime cardinality.

\begin{lem}\label{2-regular-fix-sig-1-part-1}
Use~\Cref{notn-2} and~\Cref{notn-blocks}. Suppose that $\mathcal X \succeq \mathcal B_2$. Then the orbits of $\langle \rho_2 \rangle$ form a $G$-invariant partition.
\end{lem}

\begin{proof}
Using~\Cref{orbits-blocks}, it is sufficient to show that the orbits of $\sigma_2$ are the same as the orbits of $\rho_2$. 

By~\Cref{sigmas-on-X}(\ref{X-constant}) on any block $X \in \mathcal X$ there is a constant $k_X$ such that $\sigma_1=\rho_1^{k_X}$ everywhere on $X$. In particular, since $\mathcal X \succeq \mathcal B_2$,on any block $B \in \mathcal B_2$ there is some $k_X$ such that $\sigma_1=\rho_1^{k_X}$ everywhere on $B$. By~\Cref{commuting-refined}, $\sigma_2$ commutes with $\rho_1$. 

This shows that the conditions of~\Cref{sigmas-on-X}(\ref{X-constant}) are satisfied for $i=1$, $j=2$, and $\mathcal C=\mathcal B_2$.  Thus for any $B \in \mathcal B_2$ there is a constant $k_B$ such that $\sigma_2=\rho_2^{k_B}$ everywhere on $B$. Since $B$ was arbitrary, 
 the orbits of $\langle \sigma_2 \rangle$ coincide with the orbits of $\langle \rho_2 \rangle$.
\end{proof}

\begin{lem}\label{2-regular-fix-sig-2-part-1}
Use~\Cref{notn-2} and~\Cref{notn-blocks}. If $\mathcal X \not\succeq \mathcal B_2$ then there is some $\beta \in G^{(2)}$ such that after replacing $R_r^\pi$ by $R_r^{\pi\beta}$, the new $\mathcal X$ has $\mathcal X \succeq\mathcal B_2$.
\end{lem}

\begin{proof}
From each orbit of $\langle \rho_2 \rangle$ on $\mathcal X$, choose a single representative block of $\mathcal X$. Define $\beta$ to fix every point in each of these representative blocks. If $X$ is a representative block and $X{\sigma_2}=X{\rho_2^{m_X}}$, then on $X{\sigma_2^i}$ define $\beta$ to act as $\sigma_2^{-i}\rho_2^{m_Xi}$. 

By the way we have defined $\beta$, it fixes every block of $\mathcal B_1$ (since $\mathcal X \not\succeq \mathcal B_2$, each block of $\mathcal X$ meets any block of $\mathcal B_2$ in at most one block of $\mathcal B_1$). 

Let $u,v \in \Omega$. If $v \in X_u$ then by the definition of $\beta$, there is some $i$ and some $m_{X_u}$ such that $(u,v)\beta=(u,v)\sigma_2^{-i}\rho_2^{m_{X_u}i}$. If $v \notin X_u$ then $B_{1,v}$ lies in an orbit of $G_u$. By~\Cref{in-2-closure} we conclude that there is some $h \in G$ such that $(u,v)\beta=(u,v)h$.  Thus $\beta \in G^{(2)}$.

Taking $i=2$ and on the orbit of any representative block $X$ taking $k=m_X$ in~\Cref{conj-works-2} yields $\sigma_i^\beta=\rho_i^{m_X}$ on that orbit. Thus we have the orbits of $\langle \sigma_2^\beta \rangle$ are the same as the orbits of $\langle \rho_2\rangle$, and therefore form a $G$-invariant partition $\mathcal C$ by~\Cref{orbits-blocks}.

By~\Cref{sigmas-on-X}(\ref{C-in-X}), $\mathcal X \succeq \mathcal C$; since we also have $\mathcal X \succeq \mathcal B_1$ and $\mathcal B_2$ is the smallest $R_r$-invariant partition that follows both $\mathcal B_1$ and $\mathcal C$ in our partial order, we must have $\mathcal X \succeq \mathcal B_2$.
\end{proof}

This is enough to allow us to complete the proof that $D_{2pq}$ is a CI$^{(2)}$-group; however, since our goal is to deal with $D_{2pqr}$, we will not provide a direct proof but instead will continue with additional results that will be needed for these groups.

Unfortunately, from this point on, details get very complicated and it seems necessary to restrict our attention to the case $s=3$.

\begin{lem}\label{sig1sig2difonF2}
Use~\Cref{notn-2} with $s=3$. Suppose for every $y \in F_1$ and every $k \in \{1,2,3\}$, $y\sigma_k=y\rho_k$, and that there are constants $i,j \neq 1$ such that for every $z \in F_2$, $z\sigma_1=z\rho_1^j$, $z\sigma_2=z\rho_2^i$, and $z\sigma_3=z\rho_3$. Then there is some $\beta \in G^{(2)}$ such that $R_r^{\pi\beta}=R_r$.
\end{lem}

\begin{proof}
Since $i-1 \in \mathbb Z_{p_2}^*$ it has a multiplicative inverse, say $i'$. Likewise, $j-1$ has a multiplicative inverse $j'$ in $\mathbb Z_{p_1}^*$. For any $z \in F_2$ and any $a \in \mathbb Z_{p_1}, b \in \mathbb Z_{p_2}$, we have 
$$z(\sigma_1\rho_1^{-1})^{aj'}(\sigma_2\rho_2^{-1})^{bi'}=z\rho_1^{aj'(j-1)}\rho_2^{bi'(i-1)}=z\rho_1^a\rho_2^b,$$
while for any $y \in F_1$, $y(\sigma_1\rho_1^{-1})^{aj'}(\sigma_2\rho_2^{-1})^{bi'}=y$. Thus $B_{2,z}$ lies in an orbit of $G_y$.

Let $\gamma\in R_r^\pi$ be such that $x\tau_1=x\gamma$. Define $\beta$ to fix every point of $F_1$, and for $z \in F_2$, $z\beta=z\gamma\tau_1$. Since $z\gamma\tau_1=z$ and $G$ is block-regular on $\mathcal B_2$, $\beta$ fixes every block of $\mathcal B_2$. If $u,v \in F_1$ then $(u,v)\beta=(u,v)$; if $u,v\in F_2$ then $(u,v)\beta=(u,v)\gamma\tau_1$. If $u \in F_1$ and $v \in F_2$ then by~\Cref{in-2-closure}, there is some $h \in G$ such that $(u,v)\beta=(u,v)h$. Thus $\beta \in G^{(2)}$. For $k \in \{1,2,3\}$, if $y \in F_1$ then $y\sigma_k^\beta=y\sigma_k=y\rho_k$, while if $z \in F_2$ then $$z\sigma_k^\beta=z\tau_1\gamma\sigma_k\gamma\tau_1=z\tau_1\sigma_k^{-1}\tau_1.$$ Since $z\tau_1\in F_1$, this is the same as $z\tau_1\rho_k^{-1}\tau_1=z\rho_k$. 

Thus $C_r^{\pi\beta}=C_r$, and~\Cref{cyclic-enough} completes the proof.
\end{proof}

\begin{lem}\label{p1p3-blocks}
Use~\Cref{notn-2} with $s=3$. Suppose that $G$ is block-regular on $\mathcal B_2$, that the orbits of $\langle \rho_2\rangle$ form a $G$-invariant partition $\mathcal C$, and that the orbits of either $\langle \rho_1,\rho_3\rangle$ or $\langle \rho_2,\rho_3 \rangle$ form a $G$-invariant partition $\mathcal D$. Then we can find $\beta \in G^{(2)}$ such that $R_r^{\pi\beta}=R_r$. 
\end{lem}

\begin{proof}
By~\Cref{reordering} we may exchange $p_1$ with $p_2$ if necessary, so without loss of generality let us assume that the orbits of $\langle \rho_1, \rho_3 \rangle$ form a $G$-invariant partition. Note that the intersection of any block of this partition with any block of $\mathcal B_2$ is either empty, or a single block of $\mathcal B_1$.

Use~\Cref{notn-blocks}. By~\Cref{sigmas-on-X}(\ref{C-in-X}), we have $\mathcal B_1,\mathcal C \preceq \mathcal X$, so (as in the proof of~\Cref{2-regular-fix-sig-2-part-1}), $\mathcal B_2 \preceq \mathcal X$. For any $y, z \in F_y$, we have $|C_z\cap D_y|=1$. Since $D_y$ must be fixed setwise by $G_y$, whenever $C_z$ is fixed setwise by an element of $G_y$ the point of intersection must be fixed. Therefore $C_z$ cannot lie in an orbit of $G_y$, so $y \equiv_{\mathcal C} z$. We conclude  that $\mathcal B_3 \preceq \mathcal Y$, 

There are now only two possibilities for $\mathcal Y$: $\mathcal Y=\mathcal B_3$, or $\mathcal Y=\{\Omega\}$. In either case, using~\Cref{sigmas-on-X}(\ref{X-constant}), we have $\sigma_2=\rho_2$ on $F_1$, and there is some $i$ such that $\sigma_2=\rho_2^i$ on $F_2$. In the latter case, we also have $i=1$ and $\sigma_2=\rho_2$.

Since the orbits of $\langle \rho_1,\rho_3\rangle$ are $G$-invariant, they coincide with the orbits of $\langle \sigma_1,\sigma_3\rangle$ by~\Cref{orbits-blocks}. Thus the orbits of $\rho_3$ on the blocks of $\mathcal B_1$ must coincide with the orbits of $\sigma_3$ on these blocks, so for every $B \in \mathcal B_1$ there is some $a_B$ such that $B\sigma_3=B\rho_3^{a_B}$. Since by~\Cref{notn-2} $B_{2,x}\sigma_3\rho_3^{-1}=B_{2,x}$ and $G$ is block-regular on $\mathcal B_2$, $\sigma_3\rho_3^{-1}$ must fix every block of $\mathcal B_2$, so we must have $a_B=1$ for every $B$. Thus $\sigma_3\rho_3^{-1}$ fixes every block of $\mathcal B_1$.

If $i=1$ (in particular if $\mathcal Y=\{\Omega\}$) then we have now shown that $\langle C_r, C_r^\pi\rangle$ is block-regular on $\mathcal B_1$, so by~\Cref{cyclic-enough}, so must $G$ be. Now~\Cref{cor-reg-B1} completes the proof.

We may now assume that $\mathcal Y=\mathcal B_3$ and $i \neq 1$. We consider two possibilities for the action of $\sigma_1$: either $\sigma_1=\rho_1$ on $F_1$ and there is some $j$ such that $\sigma_1=\rho_1^j$ on $F_2$, or there exist $y,z$ with $z \in F_y$ such that $y\sigma_1=y\rho_1^{j_1}$ and $z \sigma_1=z\sigma_1^{j_2}$ and $j_1 \neq j_2$.

\textbf{Case 1. $\mathbf{ \mathcal X \succeq \mathcal B_3.}$}
By~\Cref{sigmas-on-X}(\ref{X-constant}) since $x \sigma_1=x\rho_1$ we have $\sigma_1=\rho_1$ on $F_1$, and there is some $j$ such that $\sigma_1=\rho_1^j$ on $F_2$. 

Given $y \in \Omega$, let $k$ be such that $y\sigma_3\rho_3^{-1}=y\rho_1^k$, so that $\sigma_3 \rho_3^{-1}\rho_1^{-k} \in G_y$. Let $y_1, y_2, \ldots, y_\ell$ be any $\equiv_{\mathcal B_1}$-chain starting at $y$, and suppose inductively that $\sigma_3\rho_3^{-1}\rho_1^{-k} \in G_{y_i}$. Since $\rho_1$ commutes with $\sigma_3$ (by~\Cref{commuting-refined}), the fact that $B_{1,y_{i+1}}$ does not lie in an orbit of $G_{y_i}$ implies that $\sigma_3\rho_3^{-1}\rho_1^{-k}$ must fix every point of $B_{1,y_{i+1}}$, so $\sigma_3\rho_3^{-1}\rho_1^{-k} \in G_{y_{i+1}}$. Since $\mathcal X \succeq \mathcal B_3$, this implies that $\sigma_3\rho_3^{-1}\rho_1^{-k}$ fixes every point of $F_y$. Therefore $\sigma_3=\rho_3\rho_1^k$ everywhere on $F_y$.
Now by~\Cref{pth-power}, we must have $k=0$. Thus $\sigma_3=\rho_3$ on $F_y$, and since $y$ was arbitrary, everywhere. 

If $j=1$ then $\langle C_r, C_r^\pi\rangle$ is block-regular on $\mathcal C$, so by~\Cref{cyclic-enough}, $G$ is also, and~\Cref{cor-reg-B1} completes the proof. The remaining possibility is that $j \neq 1$. In this case,~\Cref{sig1sig2difonF2} completes the proof.

\textbf{Case 2. $\mathbf{ \mathcal X \not\succeq \mathcal B_3.}$} So each block of $\mathcal X$ meets each block of $\mathcal B_3$ in at most one block of $\mathcal B_2$. If the blocks of $\mathcal X$ have cardinality $2p_1p_2$ then fix $\ell$ such that $x\tau_1\rho_3^\ell$ is in the same block of $\mathcal X$ as $x$; otherwise, take $\ell=0$. 

Define $\beta_1$ to fix every point of $B_{2,x}$ and $B_{2,x} \tau_1\rho_3^\ell$. Any other point $z \in \Omega$ has a unique representation as $y\sigma_3^a$ for some $1 \le a \le p_3-1$ and some $y\in B_{2,x}\cup B_{2,x} \tau_1\rho_3^\ell$. Using this representation, define $z\beta=_1z\sigma_3^{-a}\rho_3^a$. 

Let $u, v \in \Omega$. If $v \in X_u$ then there is some $a$ such that $(u,v)\beta_1=(u,v)\sigma_3^{-a}\rho_3^a$. If $v \notin X_u$ then $B_{1,v}$ lies in an orbit of $G_u$. Observe that since $\sigma_3\rho_3^{-1}$ fixes every block of $\mathcal B_1$, so does $\beta_1$. Thus~\Cref{in-2-closure} gives us some $h \in G$ such that $(u,v)\beta_1=(u,v)h$. So $\beta_1\in G^{(2)}$. 

Since $\mathcal X \succeq \mathcal B_2$, there is some $j$ such that for every $y \in B_{2,x}\tau_1\rho_3^\ell$, $y\sigma_1=y\rho_1^j$. Also, for every $y \in B_{2,x}$ we have $y\sigma_1=y\rho_1$. Take any $z \in \Omega$, and let $a$ be such that $z=y'\sigma_3^a$ for some $y'\in B_{2,x}\cup B-{2,x}\tau_1\rho_3^{\ell}$. Note that there is some $y \in B_{2,x}\cup B-{2,x}\tau_1\rho_3^{\ell}$ such that $z=y\rho_3^a$. Now
$$z\sigma_1^{\beta_1}=z\beta_1^{-1}\sigma_1\beta_1=z\rho_3^{-a}\sigma_3^a\sigma_1\sigma_3^{-a}\rho_3^a=y\sigma_1\rho_3^a$$
and this is either $y\rho_1\rho_3^a$ (if $y \in F_1$), or $y\rho_1^j\rho_3^a$ (if $y \in F_2$), which is either $z\rho_1$ (if $z \in F_1$), or $z\rho_1^j$ (if $z \in F_2$).

The same calculations with $\sigma_2$ show that $\sigma_2^{\beta_1}=\sigma_2$, which is $\rho_2$ on $F_1$ and $\rho_2^i$ on $F_2$. Meanwhile,~\Cref{conj-works-2} shows that $\sigma_3^{\beta_1}=\rho_3$ everywhere. We can now finish the proof as before: if $j=1$ then $\langle C_r, C_r^{\pi\beta_1}\rangle$ is block-regular on $\mathcal C$, so by~\Cref{cyclic-enough}, $\langle R_r, R_r^{\pi\beta_1}\rangle$ is also, and~\Cref{cor-reg-B1} produces a $\beta_2$ that completes the proof. The remaining possibility is that $j \neq 1$. In this case,~\Cref{sig1sig2difonF2}  produces a $\beta_2$ that completes the proof.
\end{proof}

Our first couple of results in this section effectively showed that we may assume that $\mathcal X \succeq \mathcal B_2$, and the same for $\mathcal Y$ when it exists. The remaining steps in our proof largely amount to considering various possibilities for what $\mathcal X$ and $\mathcal Y$ can be. We start by dealing with several possibilities involving the blocks of each being as small as possible: each has blocks of cardinality $p_1p_2$ or $2p_1p_2$, and if both have blocks of cardinality $2p_1p_2$ then the partitioms are equal.

\begin{lem}\label{X-Ycoincide}
Use~\Cref{notn-2} with $s=3$, and~\Cref{notn-blocks}. Suppose that $G$ is block-regular on $\mathcal B_2$ and there is a $G$-invariant partition with blocks of cardinality $p_2$, so that $\mathcal Y$ is defined. Suppose that $\mathcal X, \mathcal Y \succeq \mathcal B_2$ and there is some $G$-invariant partition $\mathcal D$ with blocks of cardinality $2p_1p_2$ such that $\mathcal X,\mathcal Y \preceq \mathcal D$. Then there is some $\beta \in G^{(2)}$ such that $R_r^{\pi\beta}=R_r$.
\end{lem}

\begin{proof}
For $0 \le k \le p_3-1$, for $z \in D_x\rho_3^k$ define $z \beta=z \sigma_3^{-k}\rho_3^k$. 

Let $u, v \in \Omega$. If $v \in D_u$ then there is some $k$ such that $(u,v)\beta=(u,v)\sigma_3^{-k}\rho_3^k$. If $v \notin D_u$ then $v \notin X_u$ and $v \notin Y_u$, so both $C_v$ and $B_{1,v}$ lie in an orbit of $G_u$ (the same orbit since $v$ is in both). It is not hard to see that this forces $B_{2,v}$ to lie in this orbit of $G_u$.
Since $G$ is block-regular on $\mathcal B_2$ and by~\Cref{notn-2} $B_{2,x}\sigma_3=B_{2,x}\rho_3$, it is straightforward to deduce that $\beta$ fixes every block of $\mathcal B_2$ setwise. Using~\Cref{in-2-closure}, this gives $h \in G$ such that $(u,v)\beta=(u,v)h$. We conclude $\beta \in G^{(2)}$. 

Note that since $\beta$ acts as $\sigma_3^{-i}\rho_3^i$ for every element of $D_x\rho_3^i$, the conditions of~\Cref{conj-works-2} are satisfied for every $y \in D_x$. We conclude that $\sigma_3^\beta=\rho_3$ everywhere.

Since $\mathcal X, \mathcal Y \succeq \mathcal B_2$, using~\Cref{sigmas-on-X}(\ref{X-constant}) we conclude that for each $B\in \mathcal B_2$, there exist constants $i_{B}, j_{B}$ such that everywhere on $B$ we have $\sigma_1=\rho_1^{i_{B}}$ and $\sigma_2=\rho_2^{j_{B}}$. In particular, we have $i_{B_{2,x}}=j_{B_{2,x}}=1$, and this is also true for $\sigma_1^\beta$ and $\sigma_2^\beta$. Since $\sigma_3^\beta=\rho_3$ and $\sigma_3^\beta$ commutes with both $\sigma_1^\beta$ and $\sigma_2^\beta$, this forces $\sigma_1^\beta=\rho_1$ and $\sigma_2^\beta = \rho_2$ everywhere on $F_1$. Furthermore, on $F_2$ there is a single pair of constants $i$ and $j$ such that $\sigma_1^\beta=\rho_1^{i}$ and $\sigma_2^\beta=\rho_2^{j}$ everywhere on $F_2$.

The orbits of $\rho_1$ and $\sigma_1^\beta$ coincide, as do the orbits of $\rho_3$ and $\sigma_3^\beta$, so by~\Cref{orbits-blocks}, the orbits of $\langle \rho_1,\rho_3\rangle$
are invariant under $\langle R_r, R_r^{\pi\beta}\rangle$. Now~\Cref{p1p3-blocks} gives us a $\beta' \in G^{(2)}$ such that $R_r=R_r^{\pi\beta\beta'}$, completing the proof. 
\end{proof}

Our next result deals with the next-smallest possibility: the cardinalities of the blocks of both $\mathcal X$ and $\mathcal Y$ is $2p_1p_2$, but the partitions need not be equal. For this result, we make an additional assumption that either $\rho_1$ or $\rho_2$ commutes with every element of $C_r^\pi$. We will set aside for later the possibility that this hypothesis does not hold.

\begin{lem}\label{Ks-and-Ls-part0}
Use~\Cref{notn-2} with $s=3$, and~\Cref{notn-blocks}. Suppose that $G$ is block-regular on $\mathcal B_2$, that the orbits of $\langle \rho_2\rangle$ form a $G$-invariant partition $\mathcal C$, and that either $\rho_1$ or $\rho_2$ is central in $\langle C_r, C_r^\pi\rangle$. Further suppose that the cardinality of the blocks of $\mathcal X$ and of $\mathcal Y$ is $2p_1p_2$.

Then there is some $\beta \in G^{(2)}$ such that $R_r^{\pi\beta}=R_r$.
\end{lem}

\begin{proof}
Without loss of generality, since we can use~\Cref{reordering} to exchange $p_1$ with $p_2$, we may assume that $\rho_1$ is central in $\langle C_r, C_r^\pi\rangle$.

Let $\tau \in R_r$ be such that $\tau$ fixes $Y_x$, and let $\tau' \in R_r$ be such that $\tau'$ fixes $X_x$. Note we can choose $\tau'$ so that $\tau'=\tau\rho_3^\ell$ for some $\ell$. Then 
$$\mathcal Y =\{B_{2,x}\rho_3^i\cup B_{2,x}\tau\rho_3^i: 0 \le i \le p_3-1\}$$
$$\text{and }\mathcal X =\{B_{2,x}\rho_3^i\cup B_{2,x}\tau'\rho_3^i: 0 \le i \le p_3-1\}=\{B_{2,x}\rho_3^i\cup B_{2,x}\tau\rho_3^{i+\ell}: 0 \le i \le p_3-1\}.$$ If $\ell=0$ then~\Cref{X-Ycoincide} completes the proof, so we may assume $1 \le \ell \le p_3-1$.

Define $\beta$ as follows. If $y \in B_{2,x}\rho_3^i$ then $y\beta=y\sigma_3^{-i}\rho_3^i$. If $z \in F_2 \cap X_y$ so that $z \in B_{2,x}\tau\rho_3^{i+\ell}$, then define $j_i$ so that $C_y\beta\rho_1^{j_i}=C_y$ and $k_i$ so that $C_z\sigma_3^{-i-\ell}\rho_3^{i+\ell}\rho_1^{k_i}=C_z$. This can be done since $\rho_1$ and $\sigma_3$ commute. Now $z\beta=z\sigma_3^{-i-\ell}\rho_3^{i+\ell}\rho_1^{k_i-j_i}$.

We claim first that $\beta \in G^{(2)}$. If $v \in B_{2,u}$ then there exist $i, a$ such that $(u,v)\beta=(u,v)\sigma_3^{-i}\rho_3^i\rho_1^a$. 

If $v \notin X_u$ and $v \notin Y_u$ then both $B_{1,v}$ and $C_v$ lie in an orbit of $G_u$; since $v$ is in both of these, it is not hard to see that $B_{2,v}$ lies in this orbit of $G_u$. Note that $\beta$ has been defined to fix each block of $\mathcal B_2$. By~\Cref{in-2-closure}, there is some $h \in G$ such that $(u,v)\beta=(u,v)h$. 

If $v \in Y_u$ and $v \notin X_u$ then without loss of generality $u \in B_{2,x}\rho_3^i$ and $v \in B_{2,x}\tau\rho_3^i$ for some $i$. Now by definition of $\beta$, we have $u \beta=u\sigma_3^{-i}\rho_3^i$ and $B_{1,v}\beta=B_{1,v}\sigma_3^{-i}\rho_3^i$ (since the action of $\rho_1$ fixes every block of $\mathcal B_1$). Since $v \notin X_u$ we have $B_{1,v}$ lies in an orbit of $G_{u}$, so by~\Cref{in-2-closure} there is some $h \in G$ such that $(u,v)\beta=(u,v)h$.

Finally, if $v \in X_u$ and $v \notin Y_u$ then without loss of generality $u \in B_{2,x}\rho_3^i$ and $v \in B_{2,x}\tau\rho_3^{i+\ell}$ for some $i$. Now by definition of $\beta$, we have $C_u \beta=C_u\rho_1^{-j_i}$ and $C_v\sigma_3^{-i-\ell}\rho_3^{i+\ell}=C_v\rho_1^{-k_i}$, so $C_v\beta=C_v\rho_1^{-j_i}$. Since $v \notin Y_u$ we have $C_{v}$ lies in an orbit of $G_{u}$, so by ~\Cref{in-2-closure} there is some $h \in G$ such that $(u,v)\beta=(u,v)h$. We conclude that $\beta \in G^{(2)}$. 

Now we show that for $i \in \{1,3\}$, $\sigma_i^{\beta}=\rho_i$ on $F_1$, and on $F_2$ there is some $\ell_i$ such that $\sigma_i^\beta=\rho_i^{\ell_i}$, with $\ell_3=1$.

Since $\rho_1$ is central in $\langle C_r, C_r^{\pi}\rangle$ (and in particular commutes with $\sigma_3$), we have $\sigma_1=\rho_1$ on $F_1$ and $\sigma_1=\rho_1^{\ell_1}$ on $F_2$. For $y \in F_1$ then, $$y \beta^{-1}\sigma_1\beta=y\rho_3^{-i}\sigma_3^i\sigma_1\sigma_3 ^{-i}\rho_3^i=y\rho_3^{-i}\sigma_1\rho_3^i=y\rho_3^{-i}\rho_1\rho_3^i=y\rho_1.$$ Likewise, for $z \in F_2$, $$z \beta^{-1}\sigma_1\beta=z\rho_1^{j_i-k_i}\rho_3^{-i-\ell}\sigma_3^{i+\ell}\sigma_1\sigma_3 ^{-i-\ell}\rho_3^{i+\ell}\rho_1^{k_i-j_i}=y\rho_3^{-i-\ell}\sigma_1\rho_3^{i+\ell}=y\rho_3^{-i-\ell}\rho_1^{\ell_1}\rho_3^{i+\ell}=y\rho_1^{\ell_1}.$$ 

Also, for $y \in F_1$ with $y \in B_{2,x}\rho_3^i$, we have $$y\beta^{-1}\sigma_3\beta=y\rho_3^{-i}\sigma_3^i\sigma_3\sigma_3^{-i-1}\rho_3^{i+1}=y\rho_3^{-i}\rho_3^{i+1}=y\rho_3.$$ And for $z \in F_2$ with $z \in B_{2,x}\tau\rho_3^{i+\ell}$, we have $$z\beta^{-1}\sigma_3\beta=z\rho_1^{j_i-k_i}\rho_3^{-i-\ell}\sigma_3^{i+\ell}\sigma_3\sigma_3^{-i-\ell-1}\rho_3^{i+\ell+1}\rho_1^{k_{i+1}-j_{i+1}}$$ $$=z\rho_3^{-i-\ell}\rho_3^{i+\ell+1}\rho_1^{k_{i+1}-k_i-j_{i+1}+j_i}=z\rho_3\rho_1^{k_{i+1}-k_i-j_{i+1}+j_i}.$$ We now explain why $k_{i+1}-k_i-j_{i+1}+j_i=0$. For $1 \le i \le p_3$, define $a_i$ to be the value such that for $y \in B_{2,x}\rho_3^i$ we have $C_y\sigma_3^{-1}\rho_3=C_y\rho_1^{a_i}$. Notice that $j_i=\sum_{b=1}^i a_b$. Furthermore, when $z \in X_y$ so that $z \in B_{2,x}\tau\rho_3^{i+\ell}$ we must have $C_z\sigma_3^{-1}\rho_3=C_z\rho_1^{a_i}$, and therefore $k_i=\sum_{b=1-\ell}^{i}a_b$, where subscripts are calculated modulo $p_3$. Thus, $$k_{i+1}-k_i-j_{i+1}+j_i=a_{i+1}-a_{i+1}=0.$$ So $z\sigma_3^\beta=z\rho_3$.

The orbits of $\rho_1$ and $\sigma_1^\beta$ coincide, as do the orbits of $\rho_3$ and $\sigma_3^\beta$, so by~\Cref{orbits-blocks}, the orbits of $\langle \rho_1,\rho_3\rangle$
are invariant under $\langle R_r, R_r^{\pi\beta}\rangle$. Now~\Cref{p1p3-blocks} gives us a $\beta' \in G^{(2)}$ such that $R_r=R_r^{\pi\beta\beta'}$, completing the proof. 
\end{proof}

We now switch to considering the other end of things, where the blocks of $\mathcal X$ and $\mathcal Y$ are as large as possible. Our next result shows that if the orbits of $G_x$ include each block of $\mathcal B_2$ in $F_1$ other than $B_{2,x}$, then it is not possible for both $\mathcal X$ and $\mathcal Y$ to consist of a single block.

\begin{lem}\label{Y<Omega}
Use~\Cref{notn-2} with $s=3$, and~\Cref{notn-blocks}. Suppose that $G$ is block-regular on $\mathcal B_2$, that the orbits of $\langle \rho_2\rangle$ form a $G$-invariant partition $\mathcal C$, and that $\mathcal X, \mathcal Y \succeq \mathcal B_2$. Further suppose that every block of $\mathcal B_2$ in $F_1$ other than $B_{2,x}$ lies in an orbit of $G_x$.

If $p_2>p_1$ then $\mathcal Y \prec \{\Omega\}$; likewise, if $p_1>p_2$ then $\mathcal X \prec \{\Omega\}$.
\end{lem}

\begin{proof}
By~\Cref{reordering} we may exchange $p_1$ with $p_2$ if necessary, so without loss of generality let us assume that $p_2>p_1$ and deduce that $\mathcal Y \prec\{\Omega\}$. Let $z=x\rho_3$. Towards a contradiction, suppose that $\mathcal Y=\{\Omega\}$, so there is an $\equiv_{\mathcal C}$-chain from $x$ to $z$. Since every block of $\mathcal B_2$ in $F_1$ other than $B_{2,x}$ lies in an orbit of $G_x$, the first entry in such a chain that lies outside of $B_{2,x}$ must lie in $F_2$. Suppose that $u$ is this element. So there is some $x' \in B_{2,x}$ such that $C_u$ does not lie in an orbit of $G_{x'}$. 

If $C_u$ were to lie in an orbit of $G_x$ then there must be an element $g \in G_x$ of order $p_2$ that acts transitively on $C_u$ and therefore on the blocks of $\mathcal B_1$ in $B_{2,u}$. Since $g$ has order $p_2$, every orbit of $g$ has length $1$ or $p_2$. In particular, since there are $p_1$ blocks of $\mathcal C$ in $B_{2,u}$ and $p_2>p_1$, $g$ must fix each block of $\mathcal C$ in $B_{2,u}$. Since $g$ acts transitively on the blocks of $\mathcal B_1$ in $B_{2,u}$ and fixes each block of $\mathcal C$ in $B_{2,u}$ setwise, it must act transitively on each block of $\mathcal C$ in $B_{2,u}$. Conjugating by an appropriate element of $R_r$, we conclude that $C_u$ lies in an orbit of $G_{x'}$, a contradiction. So $C_u$ does not lie in an orbit of $G_x$ and we may as well assume that $u$ immediately follows $x$ in our chain.

By the same logic, the next entry in this chain, say $y$, must lie in $F_1$. Furthermore, we may as well assume that $y \notin B_{2,x}$, or by the logic of the preceding paragraph we could skip $u$ and $y$ in the chain and proceed immediately from $x$ to the next entry. Now by hypothesis, $C_y$ lies in an orbit of $G_x$. Therefore there is an element $g \in G_x$ of order $p_2$ that acts transitively on $C_y$. As before, since $g$ has order $p_2$ each of its orbits has length $1$ or $p_2$. Since $G$ is block-regular on $\mathcal B_2$, $g$ fixes each block of $\mathcal B_2$ setwise, so since there are $p_1<p_2$ blocks of $\mathcal C$ in any block of $\mathcal B_2$, each block of $\mathcal C$ must be fixed setwise by $g$. Consider the action of $g$ on $C_u$. Since $g \in G_x$ and $u$ immediately follows $x$ in our $\equiv_{\mathcal C}$-chain, $C_u$ cannot lie in an orbit of $G_x$, so it must be the case that $ug=u$. But then $g \in G_u$ so that $C_y$ lies in an orbit of $G_u$. This contradicts our choice of $y$ to immediately follow $u$ in our $\equiv_{\mathcal C}$-chain. 

This shows that it is not possible to form an $\equiv_{\mathcal C}$-chain from $x$ to $z$, so $\mathcal Y \neq \{\Omega\}$.
\end{proof}

Our next lemma is quite specific and technical but covers a case we will need in the following result.

\begin{lem}\label{Ks-and-Ls-part1}
Use~\Cref{notn-2} with $s=3$, and~\Cref{notn-blocks}. Suppose that $G$ is block-regular on $\mathcal B_2$, that the orbits of $\langle \rho_2\rangle$ form a $G$-invariant partition $\mathcal C$, that $\mathcal X, \mathcal Y \succeq \mathcal B_2$, and that either $\rho_1$ or $\rho_2$ is central in $\langle C_r, C_r^\pi\rangle$. Further suppose that the cardinality of blocks of $\mathcal K$ is $2p_1$ or $2p_1p_2$, the cardinality of blocks of $\mathcal L$ is $2p_2$ or $2p_1p_2$, if $K \in \mathcal K$ and $L \in \mathcal L$ then $|K\cap L|$ is not even, and if $v \in F_2 \cap K_x$ and $w \in F_2 \cap L_x$ then every block of $\mathcal B_2$ other than $B_{2,x}$, $B_{2,v}$, and $B_{2,w}$ lies in an orbit of $G_x$.

Then there is some $\beta \in G^{(2)}$ such that $R_r^{\pi\beta}=R_r$.
\end{lem}

\begin{proof}
Recall from~\Cref{K-in-X} that $\mathcal Y \succeq \mathcal K$ and $\mathcal X \succeq \mathcal L$.
Since $\mathcal X,\mathcal Y \succeq \mathcal B_2$, this forces the blocks of both $\mathcal X$ and $\mathcal Y$ to have cardinality some multiple of $2p_1p_2$. If both have blocks of cardinality $2p_1p_2$ then~\Cref{Ks-and-Ls-part0} completes the proof. So at least one of them must have blocks of cardinality a nontrivial multiple of $2p_1p_2$, which forces this partition to be $\{\Omega\}$. 
Since the conditions on $p_1$ and $p_2$ are equivalent, we assume without loss of generality that $\mathcal X=\{\Omega\}$, and therefore that $\rho_1=\sigma_1$ is central in $\langle C_r, C_r^\pi\rangle$.

Note that since $\mathcal X =\{\Omega\}$,~\Cref{Y<Omega} implies that $p_1<p_2$ and therefore that $\mathcal Y \prec \{\Omega\}$. Since the blocks of $\mathcal Y$ have cardinality a multiple of $2p_1p_2$ that is not $2p_1p_2p_3$, their cardinality must be $2p_1p_2$.

Let $\tau \in R_r$ be such that $\tau$ fixes $K_x$, and let $\tau' \in R_r$ be such that $\tau'$ fixes $L_x$. Siince the blocks of $\mathcal K$ and $\mathcal L$ are $R_r$-invariant and $G$ is block-regular on $\mathcal B_2$, there must be $G$-invariant partitions $\mathcal K'$ and $\mathcal L'$ such that $$\mathcal K\preceq\mathcal K'=\{B_{2,x}\rho_3^i\cup B_{2,x}\tau\rho_3^i: 0 \le i \le p_3-1\}=\{B_{2,x}\rho_3^i\cup B_{2,x}\rho_3^{-i}\tau: 0 \le i \le p_3-1\}$$ $$\text{and }\mathcal L\preceq\mathcal L'=\{B_{2,x}\rho_3^i\cup B_{2,x}\tau'\rho_3^i: 0 \le i \le p_3-1\}=\{B_{2,x}\rho_3^i\cup B_{2,x}\rho_3^{-i}\tau': 0 \le i \le p_3-1\}.$$
Furthermore, if $\ell$ is such that $B_{2,x}\tau'=B_{2,x}\tau\rho_3^\ell=B_{2,x}\rho_3^{-\ell}\tau$ then 
$$\mathcal L'=\{B_{2,x}\rho_3^i\cup B_{2,x}\rho_3^{-i-\ell}\tau: 0 \le i \le p_3-1\}.$$ By replacing $\rho_3$ and $\sigma_3$ by an appropriate power if necessary, we may assume without loss of generality that $\ell=1$.
Importantly, if $K \in \mathcal K$ has nonempty intersection with one of the two blocks of $\mathcal B_2$ in a block of $\mathcal K'$, then it has nonempty intersection with both, and the same is true for $\mathcal L$ with respect to $\mathcal L'$. Note also that $\mathcal Y=\mathcal K'$.

Let $a$ be such that $\sigma_3\rho_3^{-1}\rho_1^a \in G_{C_x}$. We claim that either the orbits of $\langle \rho_2,\rho_3\rangle$ are $G$-invariant, or there is some $x' \in F_1$ such that $\sigma_3\rho_3^{-1}\rho_1^a \in G_{C_{x'}}$ but $\sigma_3\rho_3^{-1}\rho_1^a\notin G_{C_{x'\tau}}$. Since $\sigma_1=\rho_1$, the action of $\sigma_3\rho_3^{-1}\rho_1^a$ is the same as the action of some power of $\rho_1$ on each block of $\mathcal C$ in $B_{2,x}\tau$. If this power is $0$, then by definition of $\mathcal L$ the action of $\sigma_3\rho_3^{-1}\rho_1^a$ must also fix each block of $\mathcal C$ in $B_{2,x}\rho_3^{-\ell}$, since this is in the same block of $\mathcal L'$ as $B_{2,x}\tau$. Repeating this argument, after $p_3$ iterations we have either concluded that $\sigma_3\rho_3^{-1}\rho_1^a$ fixes every block of $\mathcal C$, or there is some $x' \in F_1$ such that $\sigma_3\rho_3^{-1}\rho_1^a \in G_{C_{x'}}$ but $\sigma_3\rho_3^{-1}\rho_1^a\notin G_{C_{x'\tau}}$. In the former case, $a=0$ by~\Cref{pth-power} and the orbits of $\langle \rho_2,\rho_3\rangle$ are $G$-invariant. If this occurs, we complete the proof using~\Cref{p1p3-blocks}. So there must be some $x' \in F_1$ such that $\sigma_3\rho_3^{-1}\rho_1^a \in G_{C_{x'}}$ but $\sigma_3\rho_3^{-1}\rho_1^a\notin G_{C_{x'\tau}}$. Take $g'$ to be an appropriate power of $\sigma_3\rho_3^{-1}\rho_1^a$ so that $g'$ has the same action as $\rho_1$ on the blocks of $\mathcal C$ in $B_{2,x'}\tau$, and take $g$ to be a conjugate of $g'$ such that $g \in G_{C_x}$ has the same action as $\rho_1$ on the blocks of $\mathcal C$ in $B_{2,x}\tau$.

Define $\beta$ as follows. Let $a_i$ be such that on the blocks of $\mathcal C$ in $B_{2,x}\rho_3^i$, $\rho_3^{-i}\sigma_3^i$ has the same action as $\rho_1^{a_i}$. For $y \in K'_x\rho_3^i$, take $y\beta=y\rho_3^{-i}g^{a_i-a_{i-\ell}}\rho_3^i\rho_1^{-a_i}$.

We claim that $\beta \in G^{(2)}$. Let $u, v \in \Omega$.
If $v \in K'_u$, then there is some $i$ such that $(u,v)\beta=(u,v)\rho_3^{-i}g^{a_i-a_{i-\ell}}\rho_3^i\rho_1^{-a_i}$. 
If $v\notin K'_u$ and $v \notin L'_u$, then by hypothesis $B_{2,v}$ lies in an orbit of $G_u$, and since $\beta$ fixes every block of $\mathcal B_2$~\Cref{in-2-closure} produces some $h \in G$ such that $(u,v)\beta=(u,v)h$. 

The remaining possibility is that $v \in L'_u$ but $v \notin B_{2,u}$. Let $i$ be such that $u \in B_{2,x}\rho_3^i$, so $v \in B_{2,x}\tau\rho_3^{i+\ell}$. By definition of $\beta$, $C_u\beta=C_u\rho_3^{-i}g^{a_i-a_{i-\ell}}\rho_3^i\rho_1^{-a_i}$. Now, $C_u\rho_3^{-i}$ lies in $B_{2,x}$, and since $\sigma_1=\rho_1$, $g \in G_x$ fixes every block of $\mathcal C$ in $B_{2,x}$. Thus $C_u\beta=C_u\rho_1^{-a_i}$. Meanwhile, $C_v\beta=C_v\rho_3^{-i-\ell}g^{a_{i+\ell}-a_i}\rho_3^{i+\ell}\rho_1^{-a_{i+\ell}}$. We have $C_v\rho_3^{-i-\ell}$ lies in $B_{2,x}\tau$, and $g$ has the same action as $\rho_1$ on the blocks of $\mathcal C$ in $B_{2,x}\tau$, so $C_v\beta=C_v\rho_1^{a_{i+\ell}-a_i}\rho_1^{-a_{i+\ell}}=C_v\rho_1^{-a_i}$. Since $\mathcal Y=\mathcal K'$ we see that $v \notin Y_u$, so $C_v$ lies in an orbit of $G_u$. Now~\Cref{in-2-closure} produces some $h \in G$ such that $(u,v)\beta=(u,v)h$. This completes the proof that $\beta \in G^{(2)}$.

We now show that the orbits of $\langle \rho_2,\rho_3\rangle$ are invariant under $\langle R_r, R_r^{\pi\beta}\rangle$. We will use~\Cref{blocks-in-F1} with $i=2$ and $j=3$. Since $\mathcal C$ is invariant under $G$, it is also invariant under $\langle R_r, R_r^{\pi\beta}\rangle$ using~\Cref{2-closure-invariant}. The next condition holds with $\alpha=\tau_1\tau\rho_3^\ell$, by definition of $\mathcal L$. We need only show that the orbits of $\langle \rho_2,\rho_3\rangle$ in $F_1$ are invariant under $\langle C_r, C_r^{\pi\beta}\rangle$. Since $\rho_1=\sigma_1$ and the orbits of $\sigma_2$ are the blocks of $\mathcal C$, which are the orbits of $\rho_2$, it is sufficient to show that $\sigma_3^{\beta}$ fixes each orbit of $\langle \rho_2, \rho_3\rangle$ in $F_1$. Let $y \in F_1$ be arbitrary, say $y \in B_{2,x}\rho_3^i$. As calculated in the previous paragraph for $C_u$, we have $C_y\beta^{-1}=C_y\rho_1^{a_i}$, and by definition of $a_i$, this is the same as $C_y\rho_3^{-i}\sigma_3^i$. Likewise, since $C_y\beta^{-1}\sigma_3 \in B_{2,x}\rho_3^{i+1}$, by definition of $a_{i+1}$ we have $$C_y\beta^{-1}\sigma_3\beta=C_y\beta^{-1}\sigma_3\rho_1^{a_{i+1}}=C_y\beta^{-1}\sigma_3\sigma_3^{-i+1}\rho_3^{i+1}.$$ So we have $$C_y\beta^{-1}\sigma_3\beta=C_y\rho_3^{-i}\sigma_3^{i}\sigma_3\sigma_3^{i+1}\rho_3^{i+1}=C_y\rho_3.$$ Thus $\sigma_3^\beta$ fixes each orbit of $\langle \rho_2,\rho_3\rangle$.

Now with this new $G$ we have a $G$-invariant partition with blocks of cardinality $p_2p_3$, so~\Cref{p1p3-blocks} completes the proof.
\end{proof}

With the preceding results in hand, we are in position to deal with the case where the cardinality of the blocks of at least one of $\mathcal X$ and $\mathcal Y$ is a multiple of $p_1p_2p_3$.

\begin{lem}\label{X1-gives-blocks}
Use~\Cref{notn-2} with $s=3$, and~\Cref{notn-blocks}. Suppose that $G$ is block-regular on $\mathcal B_2$, and that the orbits of $\langle \rho_2\rangle$ form a $G$-invariant partition $\mathcal C$. 

Suppose that either $\mathcal Y \succeq \mathcal B_3$ or $\mathcal X \succeq \mathcal B_3$. Then we can find $\beta \in G^{(2)}$ such that $R_r^{\pi\beta}=R_r$.
\end{lem}

\begin{proof}
By~\Cref{reordering} we may exchange $p_1$ with $p_2$ if necessary, so without loss of generality let us assume that $\mathcal X \succeq \mathcal B_3$. By~\Cref{sigmas-on-X}(\ref{X-constant}), $\sigma_1=\rho_1$ on $F_1$, and there is some $i$ such that $\sigma_1=\rho_1^i$ on $F_2$. By~\Cref{sigmas-on-X}(\ref{B-in-X}) and (\ref{C-in-X}), we have $\mathcal Y \succeq \mathcal B_2$.

If there is a $G$-invariant partition with blocks of cardinality $p_2p_3$ then~\Cref{p1p3-blocks} completes the proof, so since the orbits of $\langle \rho_2, \rho_3\rangle$ are an invariant partition under $\sigma_1$ and $\sigma_2$ (as well as under $R_r$), we may assume that $\sigma_3$ does not treat these orbits as an invariant partition. In other words (since $\mathcal C$ is a $G$-invariant partition), there exist $C_1, C_2 \in \mathcal C$ and a value $j$ such that $C_2=C_1\alpha$ for some $\alpha \in \langle \rho_3\rangle$ and $\sigma_3\rho_3^{-1}\rho_1^{-j}$ fixes $C_1$ but not $C_2$. For some $u \in C_1$, let $k$ be such that $u\sigma_3\rho_3^{-1}\rho_2^{-k}\rho_1^{-j} \in G_u$.
Since $\sigma_3$ commutes with $\rho_1$ (because $\sigma_1=\rho_1$ on $F_1$ and $\sigma_1=\rho_1^i$ on $F_2$, see~\Cref{commuting-refined}), $\sigma_3\rho_3^{-1}\rho_2^{-k}\rho_1^{-j}$ must act as a $p_1$-cycle on the blocks of $\mathcal C$ in the block $B_2 \in \mathcal B_2$ that contains $C_2$. Let $g$ be some power of $\sigma_3 \rho_3^{-1}\rho_2^{-k}\rho_1^{-j}$ that has order $p_1$; note that $g$ commutes with $\rho_1$.

Since every non-transitive subgroup of a group of prime degree either fixes a single point or fixes every point (see~\Cref{Gy-affine}), and every orbit of $g$ has length $1$ or $p_1$, it must be the case that the $p_2$ blocks of $\mathcal B_1$ in $B_2$ are either all fixed by $g$, or exactly one of them is fixed by $g$.

 If the intersection of some block of $\mathcal K$ with $F_1$ has cardinality $p_1p_3$ then this generates a $G$-invariant partition with blocks of cardinality $p_1p_3$, and~\Cref{p1p3-blocks} completes the proof. So we may assume that every nonempty intersection of a block of $\mathcal K$ with $F_1$ must have cardinality $p_1, p_1p_2$, or $p_1p_2p_3$; that is, it is a block of $\mathcal B_1$, or a block of $\mathcal B_2$, or $F_1$. 

Suppose $\mathcal K \succeq \mathcal B_2$. This means that every element of $G$ that fixes one block of $\mathcal B_1$ in some block of $\mathcal B_2$ must fix every block of $\mathcal B_1$ in that block of $\mathcal B_2$. In particular, $g$ fixes every block of $\mathcal B_2$ by block-regularity, so as we have just argued must fix at least one block of $\mathcal B_1$ in each block of $\mathcal B_2$, and therefore must fix every block of $\mathcal B_1$. Let $z \in C_2$. We have $g \in G_u$, $B_{1,z}g=B_{1,z}$, and $g$ acts transitively on $B_{1,z}$. For any $y \in \Omega$  we have $B_{1,y}g=B_{1,y}$ and since $g$ commutes with $\rho_1$ we either have $g \in G_y$ fixes $B_{1,y}$ pointwise and is transitive on $B_{1,z}$, or $g$ is transitive on $B_{1,y}$. Therefore every $\equiv_{\mathcal B_1}$-chain that starts at $u$ consists entirely of points that are fixes by $g$, so it is not possible to form an $\equiv_{\mathcal B_1}$-chain from $u$ to $z$. Since we either have $u, z \in F_1$ or $u, z \in F_2$, this contradicts our assumption that $\mathcal B_3 \preceq \mathcal X$. 

We conclude that the $G$-invariant partition formed by taking intersections of blocks of $\mathcal K$ with blocks of $\mathcal B_3$, must be $\mathcal B_1$. This implies that the action of $G_x$ cannot fix any block of $\mathcal B_1$ in $F_1$ other than $B_x$, so by~\Cref{Gy-affine}, $G_x$ must act transitively on the blocks of $\mathcal B_1$ in any block of $\mathcal B_2$ in $F_1$. Furthermore, there is at most one block of $\mathcal B_2$ in $F_2$ for which $G_x$ does not act transitively on the blocks of $\mathcal B_1$ in this block.

Suppose that $\mathcal X=\mathcal B_3$. Then there must be an $\equiv_{\mathcal B_1}$-chain from $u$ to $z$. Furthermore, since $F_1$ and $F_2$ are distinct blocks of $\mathcal X$, every $y_i$ in this chain must lie in the same block of $\mathcal B_3$ as $u$ (and $z$), since $y_i \equiv_{\mathcal B_1} u$. 

Consider the blocks of $\mathcal L$. If the cardinality of these is a multiple of $p_3$, then there is some $\alpha \in C_r$ whose order is a multiple of $p_3$ such that $G_{C_u}=G_{C_u\alpha}$. Since the order of $\alpha$ is a multiple of $p_3$ (and is square-free), there is some $m$ such that $\alpha^m=\rho_3$. Therefore $G_{C_u}=G_{C_u\rho_3}$. Now by~\Cref{Gx-blocks}, the orbits of $\rho_3$ on $\mathcal C$ are $G$-invariant, meaning the orbits of $\langle \rho_2,\rho_3\rangle$ are $G$-invariant. This is a $G$-invariant partition with blocks of cardinality $p_2p_3$, so~\Cref{p1p3-blocks} completes the proof. 

Otherwise, $L_u \cap F_u$ is contained in $B_{2,u}$, so by~\Cref{Gy-affine}, $G_u$ is transitive on the blocks of $\mathcal C$ in every block of $\mathcal B_2$ except $B_{2,u}$ in $F_u=F_z$. We cannot have $z \in B_{2,u}$, so in any $\equiv_{\mathcal B_1}$ chain $y_1, \ldots, y_\ell$ from $u$ to $z$ there must be some $y_i$ that is not in $B_{2,u}$. Let $i$ be the lowest value such that $y_i \notin B_{2,u}$; that is, $B_{2,y_i} \neq B_{2,u}$. Conjugating by the element of $R_r$ that maps $u$ to $y_{i-1}$ (note that this fixes every block of $\mathcal B_2$), we see that $G_{y_{i-1}}$ must be transitive on the blocks of $\mathcal C$ in $B_{2,y_i}$. So there is an element of order $p_1$ in $G_{y_{i-1}}$ that acts as a $p_1$-cycle on the blocks of $\mathcal C$ in $B_{2,y_i}$. Since this element has order $p_1$, its action on the blocks of $\mathcal B_1$ in $B_{2,y_i}$ must fix at least one of these blocks setwise. Therefore this block of $\mathcal B_1$ lies in an orbit of $G_{y_{i-1}}$. Since $G_{y_{i-1}}$ is transitive on the blocks of $\mathcal B_1$ in $B_{2,y_i}$, all of $B_{2,y_i}$ lies in an orbit of $G_{y_{i-1}}$; in particular, $B_{1,y_{i}}$ lies in this orbit, a contradiction to the definition of an $\equiv_{\mathcal B_1}$-chain. This argument not only shows that $\mathcal X=\{\Omega\}$; it also shows that every $\equiv_{\mathcal B_1}$-chain from $u$ to $z$ must pass through the other block of $\mathcal B_3$.

Since $\mathcal X=\{\Omega\}$ we now have $\sigma_1=\rho_1$ from~\Cref{sigmas-on-X}(\ref{X-constant}), and therefore $\rho_1$ is central in $\langle C_r, C_r^\pi\rangle$. Furthermore when we pass between vertices of $F_1$ and $F_2$ in an $\equiv_{\mathcal B_1}$-chain, we must either pass between blocks of $\mathcal B_2$ that intersect the same block of $\mathcal K$, or between blocks of $\mathcal B_2$ that intersect the same block of $\mathcal L$ (or both). This is because if $y_i$ and $y_{i+1}$ are consecutive in an $\equiv_{\mathcal B_1}$-chain and $B_{2,y_{i+1}}$ does not intersect $K_{y_i}$ then $G_{y_i}$ does not fix any block of $\mathcal B_1$ in $B_{2,y_{i+1}}$ setwise, so by~\Cref{Gy-affine} the subgroup of $G_{y_i}$ that fixes $B_{2,y_{i+1}}$ is transitive on the blocks of $\mathcal B_1$ in $B_{2,y_{i+1}}$. Similarly, if $B_{2,y_{i+1}}$ does not intersect $L_{y_i}$ then the subgroup of $G_{y_i}$ that fixes $B_{2,y_{i+1}}$ is transitive on the blocks of $\mathcal C$ in $B_{2,y_{i+1}}$. So if $B_{2,y_{i+1}}$ does not intersect either $K_{y_i}$ or $L_{y_i}$ then $B_{2,y_{i+1}}$ is contained in an orbit of $G_{y_i}$, a contradiction.

We showed above that $L_u \cap F_u \subseteq B_{2,u}$. It must therefore also be the case that $L_u \cap F_u\tau_1 \subseteq B_{2,u}\tau$ for some $\tau \in R_r$.

If either $\mathcal K=\mathcal B_1$ or $\mathcal L\preceq \mathcal B_2$, we may be able to pass via an $\equiv_{\mathcal B_1}$-chain from $B_{2,u}$ to the unique block of $\mathcal B_2$ in $F_u\tau_1$ that has a nontrivial intersection with either $K_u$ or $L_u$, but this is the only block of $\mathcal B_2$ in $F_u\tau_1$ that we can pass to, and from it we can only return to $B_{2,u}$. This contradicts $\mathcal X=\{\Omega\}$. Furthermore, this is also true if the unique block of $F_u\tau_1$ that intersects $K_u$ and the unique block of $F_u\tau_1$ that intersects $L_u$ are equal. 
So it must be the case that each block of $\mathcal K$ has cardinality $2p_1$ and each block of $\mathcal L$ has cardinality either $2p_2$ or $2p_1p_2$, and the cardinality of the intersection of $K_u$ and $L_u$ is either $1$ or $p_1$.

We conclude that every block of $\mathcal B_2$ in $F_1$ other than $B_{2,x}$ lies in an orbit of $G_x$. Also, if $v$ lies in $F_2 \cap K_x$ and $w$ lies in $F_2 \cap L_x$, then every block of $\mathcal B_2$ in $F_2$ other than $B_{2,v}$ and $B_{2,w}$ lies in an orbit of $G_x$. Now~\Cref{Ks-and-Ls-part1} completes the proof.
\end{proof}

Finally, we return to the situation where the cardinality of the blocks of both $\mathcal X$ and $\mathcal Y$ is $2p_1p_2$, in order  to deal with the situation where neither $\rho_1$ nor $\rho_2$ is central in $\langle C_r, C_r^\pi\rangle$.

\begin{lem}\label{not-central}
Use~\Cref{notn-2} with $s=3$, and~\Cref{notn-blocks}. Suppose that $G$ is block-regular on $\mathcal B_2$, that the orbits of $\rho_2$ are $G$-invariant, that both $\mathcal X$ and $\mathcal Y$ have blocks of cardinality $2p_1p_2$ but these do not coincide, and that neither $\rho_1$ nor $\rho_2$ is central in $\langle C_r, C_r^\pi\rangle$. Then there is some $\beta \in G^{(2)}$ such that $\rho_2$ is central in $\langle C_r, C_r^\pi\rangle$.
\end{lem}

\begin{proof}
Since the blocks of $\mathcal X$ have cardinality $2p_1p_2$, it must be the case that there is some $z \in X_x$ with $z \notin B_{2,x}$ such that $B_{1,z}$ is not contained in an orbit of $G_x$. Since the blocks of $\mathcal X$ and $\mathcal Y$ do not coincide, $z \notin Y_x$, so $C_z$ is contained in an orbit of $G_x$. Thus for every $i$ it must be the case that $B_{1,z}\rho_2^i$ is not contained in an orbit of $G_x$. Consider the action of $G_x$ on the blocks of $\mathcal C$ in $B_{2,z}$ (note that $B_{2,z}$ is fixed setwise by $G_x$ by the block-regularity of $G$ on $\mathcal B_2$). This is a group acting on a set of cardinality $p_1$. It cannot be transitive since if it were, $B_{2,z}$ and therefore $B_{1,z}$ would lie in an orbit of $G_z$. Thus by~\Cref{Gy-affine} it must fix either a unique block, or every block, of $\mathcal C$ in $B_{2,z}$. This implies that $L_x \cap B_{2,z} \neq \emptyset$. Similarly, we can show that if $z \in Y_x$ with $z \notin B_{2,x}$ then $K_x \cap B_{2,z} \neq \emptyset$.

Note that on any block of $\mathcal B_2$, since $\sigma_1$ acts as some element of $\langle \rho_1 \rangle$ on any block of $\mathcal B_1$ and $\mathcal C$ is $G$-invariant, $\sigma_1$ acts as some fixed element of $\langle \rho_1 \rangle$ everywhere in this block of $\mathcal B_2$. Similarly, $\sigma_2$ acts as some fixed element of $\langle \rho_2\rangle$ everywhere in this block of $\mathcal B_2$.
For $0 \le i \le p_3-1$, let $a_i$ and $b_i$ be such that on $B_{2,x}\rho_3^i$ the action of $\sigma_1$ is the same as the action of $\rho_1^{a_i}$, and the action of $\sigma_2$ is the same as the action of $\rho_2^{b_i}$. Note that $a_0=b_0=1$ by our choice of $x$ in~\Cref{notn-2}.

Similarly to several of our other proofs, let $\tau \in R_r$ be such that $Y_x=B_{2,x} \cup B_{2,x}\tau$, and let $\ell$ be such that $X_x=B_{2,x}\cup B_{2,x}\tau\rho_3^\ell$.

Since $B_{2,x}\rho_3^i$ is in the same block of $\mathcal Y$ as $B_{2,x}\tau\rho_3^i$, $C_{x\tau\rho_3^i}$ does not lie in an orbit of $x\rho_3^i$, so $\sigma_2$ must have the same action on $B_{2,x}\tau\rho_3^i$ as $\rho_2^{b_i}$. Similarly, since $B_{2,x}\rho_3^{i}$ is in the same block of $\mathcal X$ as $B_{2,x}\tau\rho_3^{i+\ell}$, $\sigma_1$ must have the same action on $B_{2,x}\tau\rho_3^{i+\ell}$ as $\rho_1^{a_i}$. In other words, on $B_{2,x}\tau\rho_3^i$, $\sigma_1$ has the same action as $\rho_1^{a_{i-\ell}}$. 

Suppose that $L_x=X_x$. Then $G_x$ fixes every block of $\mathcal C$ in $X_x$. This implies that $a_{i-\ell}=a_i$, and by repeating this argument (using conjugates of $G_x$), $a_{i-j\ell}=a_i$ for every $j$. $a_j=1$ for every $j$, so $\sigma_1=\rho_1$, contradicting our hypothesis that $\rho_1$ is not central in $\langle C_r, C_r^\pi\rangle$. Similarly, $K_x=Y_x$ would imply $\sigma_2=\rho_2$, again a contradiction. Thus, $L_x$ has cardinality $2p_2$, and $K_x$ has cardinality $2p_1$.

We claim that for any $i$, there is an element of $G$ that fixes each $C_x\rho_1^j$ setwise and maps $B_{1,x}\rho_2^j$ to $B_{1,x}\rho_2^{jb_i}$, and also that there is an element of $G$ that fixes each $B_{1,x}\rho_2^j$ setwise and maps $C_x\rho_1^j$ to $C_x\rho_1^{ja_i}$. Note that the inverse of these elements does the same with $b_i$ replaced by $b_i^{-1}$ (the multiplicative inverse of $b_i$ in $\mathbb Z_{p_3}^*$) and $a_i$ replaced by $a_i^{-1}$.

On $B_{2,x}\rho_3^k$, $\sigma_1$ acts as $\rho_1^{a_k}$ and $\sigma_2$ acts as $\rho_2^{b_k}$. 
On $B_{2,x}\tau\rho_3^{k+\ell}$, $\sigma_1$ acts as $\rho_1^{a_{k+\ell-\ell}}=\rho_1^{a_k}$ and $\sigma_2$ acts as $\rho_2^{b_{k+\ell}}$. Let $\gamma_{2,k}\in R_r^\pi$ be such that $$x\rho_3^k\gamma_{2,k}=x\tau\rho_3^{k+\ell}=x\rho_3^k\tau\rho_3^{2k+\ell},$$ and consider the action of $\gamma_{2,k}\tau\rho_3^{2k+\ell}$ on $B_{2,x}\rho_3^k$. Since $|\tau\rho_3^{2k+\ell}|=2$ we have $$x\rho_3^k\gamma_{2,k}\tau\rho_3^{2k+\ell}=x\rho_3^k.$$ Also, $$x\rho_3^k\rho_1^i\rho_2^j\gamma_{2,k}\tau\rho_3^{2k+\ell}=x\rho_3^k\sigma_1^{ia_k^{-1}}\sigma_2^{jb_k^{-1}}\gamma_{2,k}\tau\rho_3^{2k+\ell}=x\rho_3^k\gamma_{2,k}\sigma_1^{-ia_k^{-1}}\sigma_2^{-jb_k^{-1}}\tau\rho_3^{2k+\ell}$$ $$=x\rho_3^k\gamma_{2,k}\rho_1^{-i}\rho_2^{-jb_k^{-1}b_{k+\ell}}\tau\rho_3^{2k+\ell}=x\rho_3^k\tau\rho_3^{2k+\ell}\rho_1^{-i}\rho_2^{-jb_k^{-1}b_{k+\ell}}\tau\rho_3^{2k+\ell}=x\rho_3^k\rho_1^i\rho_2^{jb_k^{-1}b_{k+\ell}}.$$ This implies that if the result is true for $b_{m\ell}$ then it is true for $b_{(m+1)\ell}$ where subscripts are calculated modulo $p_3$, so inductively it is true for every $b_k$.

Similarly, on $B_{2,x}\rho_3^k$, $\sigma_1$ acts as $\rho_1^{a_k}$ and $\sigma_2$ acts as $\rho_2^{b_k}$. 
On $B_{2,x}\tau\rho_3^{k}$, $\sigma_1$ acts as $\rho_1^{a_{k-\ell}}$ and $\sigma_2$ acts as $\rho_2^{b_{k}}$. Let $\gamma_{1,k} \in R_r^\pi$ be such that $$x\rho_3^k\gamma_{1,k}=x\tau\rho_3^k=x\rho_3^k\tau\rho_3^{2k},$$  and consider the action of $\gamma_{1,k}\tau\rho_3^{2k}$ on $B_{2,x}\rho_3^k$. Since $|\tau\rho_3^{2k}|=2$ we have $x\rho_3^k\gamma_{1,k}\tau\rho_3^{2k}=x\rho_3^k$. Also, $$x\rho_3^k\rho_1^i\rho_2^j\gamma_{1,k}\tau\rho_3^{2k}=x\rho_3^k\sigma_1^{ia_k^{-1}}\sigma_2^{jb_k^{-1}}\gamma_{1,k}\tau\rho_3^{2k}=x\rho_3^k\gamma_{1,k}\sigma_1^{-ia_k^{-1}}\sigma_2^{-jb_k^{-1}}\tau\rho_3^{2k}$$ $$=x\rho_3^k\gamma_{1,k}\rho_1^{-ia_k^{-1}a_{k-\ell}}\rho_2^{-j}\tau\rho_3^{2k}=x \rho_3^k\tau\rho_3^{2k}\rho_1^{-ia_k^{-1}a_{k-\ell}}\rho_2^{-j}\tau\rho_3^{2k}=x\rho_3^k\rho_1^{ia_k^{-1}a_{k-\ell}}\rho_2^{j}.$$ This implies that if the result is true for $a_{-m\ell}$ then it is true for $a_{-(m+1)\ell}$ where subscripts are calculated modulo $p_3$, so inductively it is true for every $a_k$. This completes the proof of our claim.

For each $k$, we can choose $z_k \in B_{2,x}\tau\rho_3^k$ such that $z_k \in K_{x\rho_3^k}$ and $z_k \in L_{x\rho_3^{k-\ell}}$, and since the blocks of $\mathcal K$ have cardinality $2p_1$ and the blocks of $\mathcal L$ have cardinality $2p_2$, this choice is unique. Furthermore, if $g \in G_{x\rho_3^k}$ and $B_{1,x\rho_3^k}\rho_2^jg=B_{1,x\rho_3^k}\rho_2^{jb}$ then $B_{1,z_k}g=B_{1,z_k}$ and $B_{1,z_k}\rho_2^jg=B_{1,z_k}\rho_2^{jb}$ since $B_{1,z_k}\rho_2^j \in K_{x\rho_3^k\rho_2^j}$ and $\mathcal K$ is $G$-invariant. Similarly, if $g \in G_{x\rho_3^k}$ and $C_{x\rho_3^{k+\ell}}\rho_1^jg=C_{x\rho_3^{k+\ell}}\rho_1^{ja}$ then $C_{z_{k+\ell}}g=C_{z_{k+\ell}}$ and $C_{z_{k+\ell}}\rho_1^jg=C_{z_{k+\ell}}\rho_1^{ja}$.

Define $\beta$ as follows. For $y=x\rho_3^k\rho_1^i\rho_2^j$, define $y\beta=x\rho_3^k\rho_1^{ia_k^{-1}}\rho_2^{jb_k^{-1}}$.  Now for
For $z=z_k\rho_1^i\rho_2^j$, define $z \beta = z_k\rho_1^{ia_{k-\ell}^{-1}}\rho_2^{jb_k^{-1}}$.

We show now that $\beta \in G^{(2)}$. For $(u,v)$ with $v\in B_{2,u}$, this follows from the claim we proved above. If $v \notin X_u$ and $v \notin Y_u$ then $G_u$ is transitive on $B_{2,v}$; since $\beta$ fixes each block of $\mathcal B_2$,~\Cref{in-2-closure} produces some $h \in G$ such that $(u,v)\beta=(u,v)h$. 

Suppose that $v \in X_u$ but $v \notin B_{2,u}$. Then $v \notin Y_u$. Without loss of generality, assume $u \in F_1$, say $u =x\rho_3^k\rho_1^{i_1}\rho_2^{j_1}$. Then since $v \in X_u$ but $v \notin B_{2,u}$, we have $v=z_{k+\ell}\rho_1^{i_2}\rho_2^{j_2}$. So $u\beta=x\rho_3^k\rho_1^{i_1a_k^{-1}}\rho_2^{j_1b_k^{-1}}$ and $v\beta=z_{k+\ell}\rho_1^{i_2a_k^{-1}}\rho_2^{j_2b_{k+\ell}^{-1}}$. The claim we proved above implies (after taking a conjugate of the inverse of one of the elements found there) that there is some element $g \in G$ that fixes each $B_{1,x\rho_3^k}\rho_2^j$ setwise, and maps each $C_{x\rho_3^k}\rho_1^j$ to $C_{x\rho_3^k}\rho_1^{ja_k^{-1}}$. So $ug=x\rho_3^k\rho_1^{i_1a_k^{-1}}\rho_2^{j_1}$. Since $G$ is block-regular on $\mathcal B_2$ it must be the case that $g$ fixes every block of $\mathcal B_2$.  Since $\beta$ fixes $x\rho_3^k \in B_{2,u}$ and by definition of $z_{k+\ell}$ we have $z_{k+\ell} \in L_{x\rho_3^k}$, so by the above argument 
we get $C_{z_{k+\ell}}\rho_1^jg=C_{z_{k+\ell}}\rho_1^{ja_k^{-1}}$. Thus $C_ug=C_u\beta$ and $C_vg=C_v\beta$. Since $v \notin Y_u$, $C_v$ lies in an orbit of $G_u$, so by~\Cref{in-2-closure} there is some $h \in G$ such that $(u,v)\beta=(u,v)h$.

Finally, suppose that $v \in Y_u$ but $v \notin B_{2,u}$. Then $v \notin X_u$. Without loss of generality, assume $u \in F_1$, say $u =x\rho_3^k\rho_1^{i_1}\rho_2^{j_1}$. Then since $v \in Y_u$ but $v \notin B_{2,u}$, we have $v=z_{k}\rho_1^{i_2}\rho_2^{j_2}$. So $u\beta=x\rho_3^k\rho_1^{i_1a_k^{-1}}\rho_2^{j_1b_k^{-1}}$ and $v\beta=z_{k}\rho_1^{i_2a_{k-\ell}^{-1}}\rho_2^{j_2b_{k}^{-1}}$. The claim we proved above implies (after taking a conjugate of the inverse of one of the elements found there) that there is some element $g \in G$ that fixes each $C_{x\rho_3^k}\rho_1^j$ setwise, and maps each $B_{1,x\rho_3^k}\rho_2^j$ to $B_{1,x\rho_3^k}\rho_2^{jb_k^{-1}}$. So $ug=x\rho_3^k\rho_1^{i_1}\rho_2^{j_1b_k^{-1}}$. Since $G$ is block-regular on $\mathcal B_2$ it must be the case that $g$ fixes every block of $\mathcal B_2$.  Since $\beta$ fixes $x\rho_3^k \in B_{2,u}$ and by definition of $z_{k}$ we have $z_{k} \in K_{x\rho_3^k}$, so by the above argument 
we get $B_{1,z_{k}}\rho_2^jg=B_{1,z_{k}}\rho_2^{jb_k^{-1}}$. Thus $B_{1,u}g=B_{1,u}\beta$ and $B_{1,v}g=B_{1,v}\beta$. Since $v \notin X_u$, $B_{1,v}$ lies in an orbit of $G_u$,  so by~\Cref{in-2-closure} there is some $h \in G$ such that $(u,v)\beta=(u,v)h$.

We have now shown that $\beta \in G^{(2)}$. Next we will show that $\sigma_2^\beta=\rho_2$.  Let $y=x\rho_3^{k}\rho_1^{i}\rho_2^{j} \in F_1$. Then $$y\sigma_2^\beta=y\beta^{-1}\sigma_2\beta=x\rho_3^{k}\rho_1^{ia_{k}}\rho_2^{jb_{k}}\sigma_2\beta=x\rho_3^{k}\rho_1^{ia_{k}}\rho_2^{(j+1)b_{k}}\beta$$ by definition of $b_{k}$. And this is $x\rho_3^{k}\rho_1^{i}\rho_2^{j+1}=y\rho_2$. 
Now let $z=z_k\rho_1^i\rho_2^j \in F_2$. Then $$z\sigma_2^\beta=z\beta^{-1}\sigma_2\beta=z_k\rho_1^{ia_{k-\ell}}\rho_2^{jb_k}\sigma_2\beta.$$ Since $\sigma_2$ has the same action on $B_{2,x}\tau\rho_3^k$ as $\rho_2^{b_k}$, this is $$z_k\rho_1^{ia_{k-\ell}}\rho_2^{(j+1)b_k}\beta=z_k\rho_1^i\rho_2^{j+1}=z\rho_2.$$ Thus after conjugating by $\beta$, $\rho_2$ is central in $\langle C_r, C_r^{\pi\beta}\rangle$, completing the proof.
\end{proof}

Putting the preceding results together, we are able to complete the proof when $s=3$ and $G$ is block-regular on the blocks of $\mathcal B_2$.

\begin{prop}\label{reg-on-B2-done}
Use~\Cref{notn-2} with $s=3$. Suppose that the action of $G$ is block-regular on the blocks of $\mathcal B_2$. Then there is some $\beta\in G^{(2)}$ such that $R_r^{\pi\beta}=R_r$.
\end{prop}

\begin{proof}
Use~\Cref{notn-blocks}. After applying~\Cref{2-regular-fix-sig-2-part-1} and/or~\Cref{2-regular-fix-sig-1-part-1} if necessary, we may assume that $\mathcal X \succeq \mathcal B_2$, that $\mathcal C$ exists, and (possibly applying~\Cref{2-regular-fix-sig-2-part-1} after exchanging $p_1$ and $p_2$ using~\Cref{reordering}) that $\mathcal Y \succeq \mathcal B_2$.

Since $G$ is block-regular on $\mathcal B_2$, for every $z \in F_2$ the partition $\mathcal D_z=\{\{(B_{2,x}\cup B_{2,z})\rho_3^i\}: 0 \le i \le p_3-1\}$ is $G$-invariant. Suppose that for some $\mathcal D_z$ we have $\mathcal X, \mathcal Y \preceq \mathcal D_z$. Then~\Cref{X-Ycoincide} completes the proof. This deals with the possibilities that the blocks of both $\mathcal X$ and $\mathcal Y$ have cardinality $p_1p_2$, or that one has cardinality $p_1p_2$ while the other has cardinality $2p_1p_2$, or that both have cardinality $2p_1p_2$ and they coincide.

The remaining possibilities for $\mathcal X$ and $\mathcal Y$ are: the blocks of both have cardinality $2p_1p_2$ but they do not coincide; or at least one of them has blocks whose cardinality is a multiple of $p_1p_2p_3$. In the latter case,~\Cref{X1-gives-blocks} completes the proof. 

If $\mathcal X$ and $\mathcal Y$ both have blocks of cardinality $2p_1p_2$ and either $\rho_1$ or $\rho_2$ is central in $\langle C_r,C_r^\pi\rangle$ then~\Cref{Ks-and-Ls-part0} completes the proof. If neither $\rho_1$ nor $\rho_2$ is central in $\langle C_r, C_r^\pi\rangle$ then~\Cref{not-central} allows us to find a conjugate in which $\rho_2$ is central. Now one of the previous cases applies and we can complete the proof.
\end{proof}

\section{$G$ is block-regular on $\mathcal B_3$}

When $s=3$ there are only two blocks of $\mathcal B_3$, so $G$ cannot help but be block-regular on $\mathcal B_3$. As in the situation where $G$ was block-regular on $\mathcal B_2$, the cases we need to consider largely depend on the structure of $\mathcal X$ and $\mathcal Y$ (using~\Cref{notn-blocks}).

One preliminary result about $D_{2pq}$ will be important; we will apply this to the action of $G$ on the blocks of $\mathcal C$.

\begin{lem}\label{orbits-D2pq}
Use~\Cref{notn-2} with $s=2$. Suppose that $\sigma_1=\rho_1$, and that the orbits of $G_x$ in $F_2$ have length $p_2$. Then the orbits of $G_x$ in $F_2$ are the orbits of $\rho_2$ in $F_2$.
\end{lem}

\begin{proof}
Since the orbits of $G_x$ in $F_2$ have length $p_2$,~\Cref{cyclic-enough} together with~\Cref{notn-2} implies that $\sigma_2$ has the same action as $\rho_2^k$ on the blocks of $\mathcal B_1$ in $F_2$, for some $1 <k \le p_2-1$.

For $0 \le j\le p_2-1$, define $a_j$ by $\sigma_2\rho_2^{-1}\rho_1^{-a_j}$ fixes $x\rho_2^{j(k-1)}$ (this works because $k-1$ is a unit in $\mathbb Z_{p_2}^*$). Note that since $\rho_1=\sigma_1$ commutes with $\sigma_2$, we also have $\sigma_2\rho_2^{-1}\rho_1^{-a_j}$ fixes $x\rho_2^j\rho_1^i$ for every $i$. For $0 \le j \le p_2-1$, define $b_j$ by $\sigma_2\rho_2^{-k}\rho_1^{-b_j}$ fixes $x\tau_1\rho_2^{j(k-1)}$ (and as above, also fixes $x\tau_1\rho_2^{j(k-1)}\rho_1^i$ for every $i$). Our goal is to show that $b_i=0$ for every $i$. Note that since all of our $b_i$s are exponents of $\rho_1$, calculations are always being performed modulo $p_1$ although we will often simply write equality for simplicity.

Let $z=x\tau_1$, and define $g_0=\sigma_2\rho_2^{-1}\rho_1^{-a_0}$. Then for every $i$, 
$$zg_0^i=z\rho_2^{i(k-1)}\rho_1^{b_{i-1}+\cdots+b_0-ia_0},$$
so this collection of $p_2$ points is the orbit of $z$ under $G_x$.

Now consider the orbit of $z\rho_2^{k-1}$ under $G_{x\rho_2^{k-1}}$. Since conjugating $G_x$ by $\rho_2^{k-1}$ produces $G_{x\rho_2^{k-1}}$ and translates its orbits, this must be
$$\{(z\rho_2^{k-1})\rho_2^{i(k-1)}\rho_1^{b_{i-1}+\cdots+b_0-ia_0}: 0 \le i \le p_2-1\}.$$
However, we can also calculate this orbit directly as we did the orbit of $z$ under $G_x$: taking $g_1=\sigma_2\rho_2^{-1}\rho_1^{-a_1}$, we have
$$(z\rho_2^{k-1})g_1^i=(z\rho_2^{k-1})\rho_2^{i(k-1)}\rho_1^{b_i+\cdots+b_1-ia_1}.$$
In particular, since these orbits must be equal, $b_0-a_0=b_1-a_1$; rearranging, $b_1-b_0=a_1-a_0$. 
More generally, substituting 
$$b_{i-1}+\cdots+b_0-ia_0=b_i+\cdots+b_1-ia_1$$
$$\text{ into }b_i+\cdots+b_0-(i+1)a_0=b_{i+1}+\cdots+b_1-(i+1)a_1$$
yields $b_i-a_0=b_{i+1}-a_1$, and rearranging gives $b_{i+1}-b_i=a_1-a_0$. Thus for every $i$, $b_i=b_0+i(a_1-a_0)$.

Since we must have $b_{p_2}=b_0$, and the above calculation yields $b_{p_2} \equiv b_0+p_2(a_1-a_0)\pmod{p_2}$, we must have $a_1=a_0$. This implies that $b_i=b_0$ for every $i$. By definition of $b_i$, we now have $\sigma_2=\rho_2^k\rho_1^{b_0}$ everywhere on $F_2$. Now~\Cref{pth-power} tells us that $\rho_1^{b_0}$ is the identity; that is, $b_0=0$. This completes the proof.
\end{proof}

Using this, we can complete the proof in the case where $G$ is not block-regular on $\mathcal B_2$.

\begin{prop}\label{3-reg-done}
Use~\Cref{notn-2} with $s=3$. Suppose that $G$ is not block-regular on the blocks of $\mathcal B_2$. Then there is some $\beta \in G^{(2)}$ such that $R_r^{\pi\beta}=R_r$.
\end{prop}

\begin{proof}
We also use~\Cref{notn-blocks}. By~\Cref{2-regular-fix-sig-2-part-1} and/or~\Cref{2-regular-fix-sig-1-part-1}, we may assume (possibly after conjugation by some $\beta \in G^{(2)}$) that $\mathcal C$ and $\mathcal Y$ are defined, and that $\mathcal X,Y \succeq \mathcal B_2$. 

Note that $\rho_1, \rho_2, \sigma_1, \sigma_2$ fix each block of $\mathcal B_2$, and for $y \in F_1$ we have $B_{2,y}\sigma_3=B_{2,y}\rho_3$, while for $z \in F_2$ we have $B_{2,z}\sigma_3=B_{2,z}\rho_3^k$ for some $k$. Since by~\Cref{cyclic-enough} the action of $C_r$ completely determines the action of $R_r$, if $k=1$ then $G$ is block-regular on $\mathcal B_2$. So we must have $1<k<p_3$. 

If $\mathcal X, \mathcal Y \preceq \mathcal B_3$ then every block of $\mathcal B_1$ in $F_2$ lies in an orbit of $G_x$ as does every block of $\mathcal C$ in $F_2$, so every block of $\mathcal B_2$ in $F_2$ lies in an orbit of $G_x$. By applying $\sigma_3\rho_3^{-1}$ this implies that $F_2$ is an orbit of $G_x$. Now~\Cref{use-cyclic} completes the proof. 

We may now assume without loss of generality that $X_x \cap F_2 \neq \emptyset$. Since there exist $i,j$ such that $g=\sigma_3\rho_3^{-1}\rho_1^i\rho_2^j \in G_x$ so fixes $X_x$ setwise, we conclude that $X_x$ intersects nontrivially with every block of $\mathcal B_2$ in $F_2$. Since $\mathcal X \succeq \mathcal B_2$, this means $F_2 \subset X_x$, and therefore since $\{\Omega\}$ is the smallest $R_r$-invariant partition that contains $x$ and $F_2$, $\mathcal X=\{\Omega\}$. By~\Cref{sigmas-on-X}(\ref{X-constant}), $\sigma_1=\rho_1$.

Since $\mathcal X=\{\Omega\}$, there must exist $\equiv_{\mathcal B_1}$-chains that pass from $F_1$ to $F_2$. Thus there must be some $z \in F_2$ such that $B_{1,z}$ does not lie in an orbit of $G_x$. 

Suppose that the blocks of $\mathcal C$ in $B_{2,z}$ all lie in a single orbit of the action of $G_x$ on the blocks of $\mathcal C$. Then there is an element $h \in G_x$ of order $p_1$ that fixes $B_{2,z}$ setwise and acts as a $p_1$-cycle on the blocks of $\mathcal C$ in $B_{2,z}$. Consider the action of $h$ on the blocks of $\mathcal B_1$ in $B_{2,z}$. Its orbits must have length $1$ or $p_1$. Using~\Cref{Gy-affine}, they either all have length $1$, or there is one orbit of length $1$ and the rest have length $p_1$. Since $h$ does not fix any block of $\mathcal C$ in $B_{2,z}$, if it fixes some block $B \in \mathcal B_1$ with $B \subset B_{2,z}$, then $B$ lies in an orbit of $G_x$. Since $B_{1,z}$ does not lie in an orbit of $G_x$, it follows that $B_{1,z}h \neq B_{1,z}$, so $h$ fixes a unique block $B_{1,z'}$ of $\mathcal B_1$ in $B_{2,z}$. Note that by the action of $h$, $B_{1,z'}$ lies in an orbit of $G_x$. Since $B_{1,z}$ does not, the blocks of $\mathcal B_1$ in $B_{2,z}$ cannot lie in a single orbit of $G_x$, so by~\Cref{Gy-affine}, every element of $G_x$ must fix $B_{1,z'}$. Thus $G_{B_{1,x}}=G_{B_{1,z'}}$ and therefore $K_x=K_{z'}$. Applying $g \in G_x$ to $z'$, we see that $K_x$ meets every block of $\mathcal B_2$ in $F_2$. Since $\mathcal Y \succeq \mathcal K$ by~\Cref{K-in-X}, this implies $\mathcal Y=\{\Omega\}$. By~\Cref{sigmas-on-X}(\ref{X-constant}), $\sigma_2=\rho_2$. However, by~\Cref{commuting}, $h$ commutes with $\rho_2$, which contradicts what we determined about the action of $h$ above (that it fixes $B_{1,z'}$ setwise but does not fix $B_{1,z}$ although there is some $b$ such that $B_{1,z}=B_{1,z'}\rho_2^b$. We conclude that the blocks of $\mathcal C$ in $B_{2,z}$ cannot all lie in a single orbit of the action of $G_x$ on the blocks of $\mathcal C$.

Conjugating by various powers of $g$, we conclude that for every $z \in F_2$, the blocks of $\mathcal C$ in $B_{2,z}$ do not all lie in a single orbit of the action of $G_x$ on the blocks of $\mathcal C$.
Since $\sigma_1=\rho_1$ commutes with every element of $G_x$ by~\Cref{commuting}, this implies that each orbit of $G_x$ on the blocks of $\mathcal C$ in $F_2$ has length $p_3$.

Applying~\Cref{orbits-D2pq} to the action of $G_{\mathcal C}$, we see that the orbits of $G_x$ on the blocks of $\mathcal C$ in $F_2$ are the orbits of $\rho_3$ in $F_2$. Conjugating by $\tau_1$ shows that the orbits of $G_{x\tau_1}$ on the blocks of $\mathcal C$ in $F_1$ are the orbits of $\rho_3$ in $F_1$. Together, these imply that the orbits of $\langle \rho_2,\rho_3\rangle$ are the same as the orbits of $\langle \sigma_2,\sigma_3\rangle$, so these orbits form a $G$-invariant partition with blocks of cardinality $p_2p_3$. Furthermore, $\rho_2, \rho_3, \sigma_2,$ and $\sigma_3$ all fix each of these orbits setwise, while $\rho_1=\sigma_1$, so the action of $G$ on this partition is block-regular. After reordering our primes as $p_2, p_3, p_1$ using~\Cref{reordering},~\Cref{reg-on-B2-done} completes the proof.
\end{proof}

Putting our results together gives our main theorem.

\begin{proof}[Proof of~\Cref{main}]
After applying~\Cref{better-blocks}, we may use~\Cref{notn-2} with $s=3$. Let $i \in 1,2, 3$ be as small as possible so that the action of $G$ is block-regular on $\mathcal B_i$. If $i=1$ then~\Cref{cor-reg-B1} shows that there is some $\beta \in G^{(2)}$ such that $R_r^{\pi\beta}=R_r$. If $i=2$ then we can reach the same conclusion from~\Cref{reg-on-B2-done}, and if $i=3$ then~\Cref{3-reg-done} yields this conclusion. In each case,~\Cref{Babai-conj} completes the proof.

Since every subgroup of a DCI-group is a DCI-group, it follows that the dihedral group of order $2pq$ is a DCI-group.
\end{proof} 

While it may be possible to push these techniques farther, to prove the result is true for $4$ or $5$ distinct primes, it should be clear that these methods become increasingly complex with more primes. I believe that a dihedral group of squarefree order that is a multiple of an odd number of primes is likely to be a DCI group, but some new ideas will be needed if this approach is ever to be successful in proving this.

\bibliography{References}{}

\providecommand{\MR}{\relax\ifhmode\unskip\space\fi MR }
\providecommand{\MRhref}[2]{%
  \href{http://www.ams.org/mathscinet-getitem?mr=#1}{#2}
}
\providecommand{\href}[2]{#2}

\end{document}